\documentclass[12pt,reqno]{amsart}
\usepackage{graphicx}
\usepackage{color}
\usepackage{amsmath,amssymb,amsthm,amsfonts}
\usepackage{hyperref}
\usepackage{mathdots}
\usepackage{graphicx}
\usepackage{color}
\usepackage{multicol}
\def\R{\mathbb{R}}

\makeatletter

\newcommand{\Rmnum}[1]{\expandafter\@slowromancap\romannumeral #1@}
\makeatother

\newtheorem{thm}{Theorem}[section]

\newtheorem{lemma}[thm]{Lemma}
\newtheorem{remark}[thm]{Remark}
\newtheorem{theorem}[thm]{Theorem}

\newtheorem{corollary}[thm]{Corollary}

\usepackage[numbers,sort&compress]{natbib}
\allowdisplaybreaks

\begin{document}
\author{Jiawei Chu}
\address{School of Mathematics, South China University of Technology, Guangzhou 510640, China}
\email{201920127901@mail.scut.edu.cn}
\author{Hai-Yang Jin}
\address{School of Mathematics, South China University of Technology, Guangzhou 510640, China}
\email{mahyjin@scut.edu.cn}

\author{Tian Xiang$^*$}
\address{Institute for Mathematical Sciences, Renmin University of China, Bejing, 100872, China}
\email{txiang@ruc.edu.cn}
\thanks{$^*$ Corresponding author.}

\title[Convection effects in  tumor  angiogenesis model]{Global dynamics in a chemotaxis model describing tumor  angiogenesis with/without mitosis in any dimensions}

\begin{abstract}
In this work, we study the following  Neumann-initial boundary value problem for a three-component chemotaxis model describing tumor  angiogenesis:
\begin{equation*}\label{system}
    \begin{cases}
    u_t=\Delta u-\chi\nabla\cdot(u\nabla v)+\xi_1\nabla\cdot(u\nabla w)+u(a-\mu u^\theta), &x\in\Omega, t>0,\\
    v_t=d\Delta v+\xi_2\nabla\cdot(v\nabla w)+u-v, &x\in\Omega, t>0,\\
    0=\Delta w+u-\bar{u}, \  \   \int_\Omega w=0, \  \  \bar{u}:=\frac{1}{|\Omega|}\int_\Omega u,  & x\in\Omega, t>0,\\
    \frac{\partial u}{\partial\nu}=\frac{\partial v}{\partial\nu}=\frac{\partial w}{\partial\nu}=0, &x\in\partial\Omega, t>0,\\
    u(x,0)=u_0(x),~v(x,0)=v_0(x),&x\in\Omega,
    \end{cases}
\end{equation*}
in a bounded smooth but not necessarily convex domain $\Omega\subset\R^n(n\geq 2)$  with  model parameters $\xi_1,\xi_2, d, \theta>0,a, \chi, \mu \geq 0$. Based on subtle energy estimates, we first identify two positive constants $\xi_0$ and $\mu_0$ such that the above problem allows only  global   classical solutions  with qualitative bounds provided one of the following conditions holds:
\begin{equation*} (1) \ \xi_1\geq \xi_0\chi^2; \ \ \  \  (2)\   \theta=1, \ \mu\geq\max\left\{1, \ \chi^\frac{8+2n}{5+n}\right\}\mu_0\chi^\frac{2}{5+n}; \ \ \  \ (3)\  \theta>1, \mu>0.
\end{equation*}
 Then, due to the obtained qualitative bounds, upon deriving higher order gradient estimates, we show  exponential convergence of bounded solutions to the spatially homogeneous equilibrium    (i) for  $\mu$ large relative to $\chi^2+\xi_1^2$ if $\mu>0$, (ii) for $d$ large if  $a=\mu=0$ and (iii) for merely  $d>0$ if $\chi=a=\mu=0$. As a direct consequence of our findings, all solutions to the above system with $\chi=a=\mu=0$ are globally bounded and they converge to constant equilibrium, and therefore, no patterns can arise.
\end{abstract}



\maketitle

\numberwithin{equation}{section}

%%%%%%%%%%%%%%%%%%%%%%%%%%%%%%%%%%
%%       Section
%%%%%%%%%%%%%%%%%%%%%%%%%%%%%%%%%%%%%

\section{Introduction and statement of main results}

To describe the branching of capillary sprouts during angiogenesis, Orme \& Chaplain \cite{OC-IMA-1996} proposed the following reaction-advection-diffusion system
\begin{equation}\label{PKS}
\begin{cases}
u_t=d_1\Delta u-\chi \nabla \cdot(u\nabla v)+\xi_1\nabla \cdot( u \nabla w),\\
v_t=d_2\Delta v+\xi_2 \nabla \cdot(v\nabla w)+\alpha u-\beta v,\\
w_t=d_3\Delta w+\gamma u-\delta w,
\end{cases}
\end{equation}
with positive parameters $d_1,d_2,d_3,\chi,\xi_1, \xi_2,\beta,\delta,\alpha,\gamma$,
where $u,v$ and $w$ denote the density of endothelial cells (ECs), adhesive sites, and the matrix (including fibronectin, laminin, and collagen IV), respectively.  Different from the classical mathematical models of  tumor  angiogenesis with chemotaxis as the principle mechanism of cell motion (\cite{CS-IMA-1993,SL-JTB-1991}), the model \eqref{PKS} was proposed  based on the experimental observations that during angiogenesis process ECs secrete a  matrix  consisting of fibronectin, laminin and collagen IV \cite{PK-CROH-1989} and the movement of  ECs is  effected by the distribution of adhesive sites  on this matrix. More specifically, the following two processes are essentially incorporated in  \eqref{PKS} (see \cite{OC-IMA-1996,PK-CROH-1989}):
\begin{itemize}
\item ECs secret  matrix and adhesive sites;
\item The spreading of  matrix with the convection of ECs and adhesive sites with it.
\end{itemize}
That is,  the movement of   ECs is governed by a combination of random motility, chemotaxis and convection.

Before proceeding to  our motivation and main results, we first recall some most relevant results to the system \eqref{PKS} under homogeneous Neumann boundary conditions (IBVP) and nonnegative initial conditions.
\begin{itemize}
\item[(i)] $\xi_1=\xi_2=0$: In this case, the first two components of  \eqref{PKS} reduce to the well-known classical (minimal) Keller-Segel chemotaxis model:
\begin{equation}\label{KS}
\begin{cases}
u_t=d_1\Delta u-\chi \nabla \cdot(u\nabla v),\\
v_t=d_2\Delta v+\alpha u-\beta v,\\
\end{cases}
\end{equation}
whose solution behaviors have been extensively studied  in various perspectives in the past five decades including boundedness, blow-up, large time behavior and pattern formation. One can find more details from  survey articles \cite{Horstmann-D, BBTW-2015, AT21-AAM} and the references therein. More precisely,  the boundedness and blowup of solutions for \eqref{KS} have been established in two or higher dimensions \cite{Nagai-Funk, W-JMPA-2013,W-JDE-2010,horstmann2001blow} based on the following Lyapunov energy functional:
\begin{equation*}\label{F1}
\mathcal{E}_1(u,v)=d_1\int_\Omega u\ln u -\chi \int_\Omega uv+\frac{\beta\chi}{2\alpha}\int_\Omega v^2+\frac{\chi d_2}{2\alpha}\int_\Omega |\nabla v|^2.
\end{equation*}
\item[(ii)] $\xi_2=0$: This case means that the convection effect of matrix on the adhesive sites is neglected, and then
the system \eqref{PKS} reduces to  the following widely studied attraction-repulsion Keller-Segel (ARKS) model
\begin{equation}\label{ARKS}
\begin{cases}
u_t=d_1\Delta u-\chi \nabla \cdot(u\nabla v)+\xi_1\nabla \cdot( u \nabla w),\\
v_t=d_2\Delta v+\alpha u-\beta v,\\
w_t=d_3\Delta w+\gamma u-\delta w.
\end{cases}
\end{equation}
The ARKS model \eqref{ARKS} has been proposed to describe the aggregation of {\it Microglia} in Alzheimer's disease in   \cite{M-A-L-A} and to describe  quorum effect  in chemotaxis \cite{P-H-2002}. In one dimensional space, the existence of global boundedness of classical solution \cite{LW-JBD-2012,JW-M2AS-2015} and time-periodic patterns and steady states patterns   \cite{LSW-DCDSB-2013} were established. In high dimensional spaces ($n\geq 2$), it has been found  that the sign of  $\Theta:=d_2\xi_1\gamma -d_3\chi\alpha$ plays an important role in determining the solution behavior of \eqref{ARKS}. More precisely, if $\Theta\geq 0$ (i.e., repulsion dominates or cancel attraction), the 2D fully parabolic ARKS model \cite{J-JMAA-2015,LT-M2AS-2015} or higher D parabolic-elliptic-elliptic simplification  of the ARKS model \cite{TW-M3AS-2013}  admits only global bounded solutions. However, if $\Theta< 0$ (i.e., attraction prevails over repulsion), based on the availability of Lyapunov functional, a critical mass phenomenon has been found (see \cite{LL-NARWA-2016,Espejo2014,JW-JDE-2016} for more details). Recently, in the repulsion dominated case, i.e,  $\Theta\geq 0$, the global  stability of constant steady state has been studied in \cite{Lin-M3AS-2018,JW-preprint}.
\item[(iii)] $\xi_1,\xi_2>0$: Due to the strong coupling of chemotaxis and convection in a cascade-like manner, which increases the complexity of mathematical analysis, the Lyapunov functional as constructed for the system \eqref{ARKS} does not work anymore. To the best of our knowledge,  the existing results on the system \eqref{PKS} seem at a rather rudimentary stage: in one dimensional space, the global existence of classical solution was very recently established  in \cite{Tao-JMAA} based on  semigroup estimate technique.  Furthermore,  based on an appropriate energy functional,  the  1D bounded solution was shown to converge to constant steady state under some restrictions on the model parameters like $\xi_2$ is small \cite{JX-CMS-2021}. Very recently, for a parabolic-parabolic-elliptic simplified model in a bounded {\it convex} domain $\Omega \subset \R^n(1\leq n\leq 3)$, Tao \& Winkler \cite{TW-NA-2021} used a  Moser-type iteration to derive the global boundedness of classical solution for large $\xi_1$ without qualitative information.
\end{itemize}
In summary of the above related results,  some interesting questions naturally arise:
\begin{itemize}
\item[(Q1)] It follows from  \cite{TW-NA-2021}  that large repulsive convection prevents any blow-up phenomenon in $\leq 3$D  (lower dimensional) {\it convex} domains. Hence, it is nature to ask whether or not large repulsive convection can prevent blow-up of solution in any dimensional bounded  smooth but not necessarily convex domains.
\item[(Q2)] In one dimensional space, due to nice Sobolev embeddings,  a small $\xi_2(\in(0,1])$ -independent  upper bound of solution is available. This makes  the small   $\xi_2$-global stability toward constant equilibrium  possible \cite{JX-CMS-2021}. However, in higher dimensions,  solution bounds typically depend  on $\xi_2$ (actually with a complex relation containing $\xi_2$ and its inverse $\xi_2^{-1}$, cf. \eqref{v-infty-bdd} and \eqref{gradw-infty-bdd} for instance), and thus the method used in \cite{JX-CMS-2021} does not work anymore. Hence,  it is challenging to study long time dynamics of bounded solutions in higher dimensions.

\item[(Q3)] Although there may be no significant increase in the rate of ECs mitosis during the first stages of angiogenesis, the mitosis occurs after the first spouts have formed \cite{OC-IMA-1996}. Moreover,  cell division also plays an essential role when repairing and remodelling of large wounds \cite{PK-CROH-1989}. Thus,  it is   interesting and practically needed  to study the effect of mitosis  for the system \eqref{PKS}.
\end{itemize}
To study the impact of  convection  more deeply and to provide relatively complete answers for  the three questions above,  for simplicity and clarity, based on the assumption that matrix diffuses much faster than adhesive sites and endothelial cells, we shall use a quasi-stationary approximation procedure as in \cite{JL-TAMS-1992, W-JMAA-2011} ($\tilde{w}=w-\bar{w}$, and then the $w$-equation becomes $d_3^{-1}\tilde{w}_t=\Delta \tilde{w}+\gamma d_3^{-1}(u-\bar{u})-\delta d_3^{-1} \tilde{w}$; then assuming $\gamma$  has the same order as $d_3$ and $\delta$ has lower order than $d_3$, and finally,  sending $d_3\rightarrow \infty$ and dropping the tilde notation) to lead to the following version of parabolic-parabolic-elliptic problem:
\begin{equation}\label{system}
    \begin{cases}
    u_t=\Delta u-\chi\nabla\cdot(u\nabla v)+\xi_1\nabla\cdot(u\nabla w)+u(a-\mu u^\theta), &x\in\Omega, t>0,\\
    v_t=d\Delta v+\xi_2\nabla\cdot(v\nabla w)+u-v, &x\in\Omega, t>0,\\
    0=\Delta w+u-\bar{u}, \  \   \int_\Omega w=0, \  \  \bar{u}:=\frac{1}{|\Omega|}\int_\Omega u, & x\in\Omega, t>0,\\
    \frac{\partial u}{\partial\nu}=\frac{\partial v}{\partial\nu}=\frac{\partial w}{\partial\nu}=0, &x\in\partial\Omega, t>0,\\
    u(x,0)=u_0(x),~v(x,0)=v_0(x),&x\in\Omega,
    \end{cases}
\end{equation}
where $\Omega\subset\mathbb{R}^n$ is a bounded domain with smooth boundary. Here we keep the parameters $\chi,\xi_1,\xi_2$ as above and simplify other parameters in an obvious way for convenience. The kinetic term  $u(a-\mu u^\theta)$ with $a,\mu\geq 0, \theta>0$ is incorporated  to see the  effect of  ECs mitosis. We shall henceforth assume that
\begin{equation}\label{initial}
(u_0, v_0)\in C^0(\bar{\Omega})\times W^{1,\infty}~~\text{with}~~u_0\geq0, v_0\geq 0,~\text{and}~u_0\not\equiv0.
\end{equation}
Then our main findings on qualitative  boundedness and convergence are stated as follows.
\begin{theorem}[{\bf Qualitative boundedness}]\label{GB}
Let  $\Omega\subset\mathbb{R}^n(n\geq 2)$ be a bounded and smooth domain, the model parameters $\xi_1,\xi_2,d,\theta>0, a,\chi, \mu\geq0$, and,  let   the initial data $(u_0,v_0)$ satisfy the regularity  \eqref{initial}. Then  there exist positive constants $\xi_0$ and $\mu_0$ depending on $d, \xi_2,  n, u_0, v_0, \Omega$  such that the IBVP \eqref{system}  admits a unique  global classical solution
provided one of the following conditions holds:
\begin{itemize}
\item (1) $\xi_1\geq \xi_0\chi^2$; \ \  \  (2) $\theta=1, \ \mu\geq \max\left\{1, \ \chi^\frac{8+2n}{5+n}\right\}\mu_0\chi^\frac{2}{5+n}$; \ \  \ (3) $\theta>1, \mu>0$.
\end{itemize}
Moreover,  the global solution is qualitatively  bounded in the following way:
\begin{equation}\label{GB-1}
\|u(\cdot,t)\|_{L^\infty}+\|v(\cdot,t)\|_{W^{1,\infty}}+\|w(\cdot,t)\|_{W^{2,\infty}}\leq M, \quad \quad \quad  \forall t>0,
\end{equation}
where, up to a multiplier depending only  on $n,u_0, v_0$ and $\Omega$,   the positive constant $M$ is explicitly expressible in terms of the model parameters $d,\chi_1, \xi_1,\xi_2,  \mu$, see Lemmas \ref{b-v}, \ref{high-reg},  \ref{b-u} and \ref{ev}. In particular, the qualitative bounds for $\|v\|_{L^\infty}$ and $\|\nabla w\|_{L^\infty}$, crucial to derive large time behaviors of bounded solutions  for \eqref{system}, are bounded as follows:  \begin{equation}\label{v-infty-bdd}
\begin{split}
\|v(\cdot,t)\|_{L^\infty}&\leq K_1\begin{cases} \left(1+\frac{1}{\xi_2}\right)\left(1
+\xi_2+(\frac{1}{d})^{\frac{n}{2}}\xi_2^{1+\frac{n}{2}}\right), &\text{ if  }  \mu=0, \\[0.25cm]
\left(1+(\frac{1}{\mu})^\frac{1}{\theta}+\frac{1}{\xi_2}\right)\left(1
+(\frac{1}{\mu})^\frac{1}{\theta}\xi_2+(\frac{1}{d})^{\frac{n}{2}}[(\frac{1}{\mu})^\frac{1}{\theta}\xi_2]^{1+\frac{n}{2}}\right), &\text{ if  }  \mu>0,
\end{cases}\\
& := M_0,
\end{split}
\end{equation}
and
\begin{equation}\label{gradw-infty-bdd}
\|\nabla w(\cdot,t)\|_{L^\infty}\leq K_2\left(1+\left(1+d_\Omega M_0^{2(n+1)}\right)d^{-1}\chi^2M_0^{1-n}+M_1^c(n)\right)^\frac{1}{n+1},
\end{equation}
where $M_1^c(n)=M_1^c$ is defined by
\begin{equation}\label{m1c-def}
M_1^c(n)=\begin{cases}
\xi_1, \quad \quad \quad \quad \quad \quad  \quad  \text{if  } \mu=0, \ \xi_1\geq \xi_0\chi^2, \\[0.25cm]
M_\mu(1), \quad \quad \quad \quad  \quad    \text{if  }  \theta=1, \ \mu> \max\left\{1, \ \chi^\frac{8+2n}{5+n}\right\}\mu_0\chi^\frac{2}{5+n}, \\[0.25cm]
M_\mu(\theta)+\frac{(\theta-1)}{\mu^\frac{n+2}{\theta-1}}
\left[(1+\frac{1}{d^{n+2}})(1+M_0\xi_2)M_0\chi^2 \right]^\frac{n+1+\theta}{\theta-1}, \  \text{if  } \theta>1,   \mu>0,
\end{cases}
\end{equation}
the function $M_\mu$ and the symbol $d_\Omega=d 1_\Omega$ with $1_\Omega$ being the indicator whether $\Omega$ is non-convex  are defined by
\begin{equation}\label{d-omega}
M_\mu(\theta)=\left(1+\xi_1(\frac{1}{\mu})^\frac{1}{\theta}+(\frac{1}{\mu})^\frac{1}{\theta}\right)(\frac{1}{\mu})^\frac{n+1}{\theta}, \ \   d_\Omega=d 1_\Omega=\begin{cases}0, &\text{ if } \Omega \text{ is convex},\\
d, &\text{ if } \Omega \text{ is non-convex}.
\end{cases}
\end{equation}
\end{theorem}
\begin{remark}When $\Omega\subset \mathbb{R}^1$ is an open interval, using simpler arguments than those of \cite{JX-CMS-2021,Tao-JMAA}, one can easily obtain global boundedness without any parameter restrictions. For fixed parameters and initial data, we  also mention that the infimums of $\xi_0$ and $\mu_0$ depend indeed on $\frac{n}{2}$ instead of $n$, see Lemma \ref{b-u}; this is comparable to the widely known $L^{\frac{n}{2}+}$-criterion \cite{BBTW-2015,Xiangjde}.   Moreover, we note that the upper bounds for $\|v\|_{L^\infty}$ and $\|\nabla w\|_{L^\infty}$, cf. \eqref{v-infty-bdd} and \eqref{gradw-infty-bdd}, are  bounded for large $d$ and, in particular, it can be non-increasing in $d$ in the case that $\Omega$ is convex. This makes the study of global stability  possible in the case of $a=\mu=0$.
\end{remark}
Thanks in particular to the qualitative bounds for $\|v\|_{L^\infty}$ and $\|\nabla w\|_{L^\infty}$ in \eqref{v-infty-bdd} and \eqref{gradw-infty-bdd}, upon successfully deriving higher order gradient estimates, we are able to prove  convergence and exponential convergence rate of bounded solutions in the following ways.
\begin{theorem}[{\bf Global stability}]\label{LT} The global bounded classical solution $(u,v,w)$ obtained  from  Theorem \ref{GB} enjoys the following convergence properties:
\begin{itemize}
\item[(C1)] In the case of  $a=\mu=0$, there exists $d_0(\chi)\geq 0$ with $d_0(0)=0$ depending on $n,u_0, v_0, \xi_1,\xi_2,\chi$ such that
whenever $d\geq d_0(\chi)$ with $\chi>0$, the solution $(u,v,w)$ converges uniformly to $(\bar{u}_0,\bar{u}_0,0)$:
\begin{equation}\label{conv-main-1}
\|u(\cdot,t)-\bar{u}_0\|_{L^\infty}+\|v(\cdot,t)-\bar{u}_0\|_{L^\infty}+\|w(\cdot,t)\|_{W^{2,\infty}}\to 0 \ \ \mathrm{as}\ \  t\to\infty.
\end{equation}
If, in addition,  $d>d_0(\chi)$, the above convergence is exponential: for some  $K_3, \zeta>0$,
 \begin{equation}\label{conv-main-2}
\|u(\cdot,t)-\bar{u}_0\|_{L^\infty}+
\|v(\cdot,t)-\bar{u}_0\|_{L^\infty}+\|w(\cdot,t)\|_{W^{2,\infty}}\leq K_3 e^{-\zeta t}, \ \ \forall t>0.
\end{equation}
    \item[(C2)] In the case of  $\mu>0$, let $C_p$ denote the Poincar\'{e} constant defined by
   \begin{equation}\label{poincare-cont}
    C_p^{-2}=\inf\left\{\int_\Omega |\nabla w|^2: \int_\Omega w=0, \int_\Omega w^2=1\right\}
\end{equation}
and let
 \begin{equation}\label{lambda-def}
\Lambda(z)=\frac{d\chi^2  +d^2C^2_p\xi_1^2
+C^2_p\xi_2^2z}{2d^2a^\frac{\theta-2}{\theta}}.
\end{equation}
   Then,  whenever
    \begin{equation}\label{mu-large-lt-intro}
\mu>\max\left\{\Lambda^\frac{\theta}{2}\left(\sup_{\mu>1}M_0^2\right),\  \Lambda^\frac{\theta}{6+n}\left(\sup_{0<\mu\leq 1}\left(M_0^2\mu^\frac{4+n}{\theta}\right)\right)\right\},
\end{equation}
 the solution $(u,v,w)$ converges exponentially  to  $((\frac{a}{\mu})^\frac{1}{\theta},(\frac{a}{\mu})^\frac{1}{\theta},0)$: for some  $K_4, \eta>0$,
 \begin{equation}\label{cv-uvw-main3}
\|u(\cdot,t)-(\frac{a}{\mu})^\frac{1}{\theta}\|_{L^\infty}+\|v(\cdot,t)-(\frac{a}{\mu})^\frac{1}{\theta}\|_{L^\infty}+\|w(\cdot,t)\|_{W^{2,\infty}}\leq K_4 e^{-\eta t}, \ \ \ \forall t>0.
\end{equation}
\end{itemize}
\end{theorem}
In the above texts, we have used the following short notations like
$$
\|f(\cdot,t)\|_{L^p}=\left(\int_\Omega |f(x,t)|^pdx\right)^\frac{1}{p}=\left(\int_\Omega |f(\cdot,t)|^p\right)^\frac{1}{p}.
$$
\begin{remark}
\
\
\begin{itemize}
\item[(i)] Compared to the global boundedness  in \cite{TW-NA-2021}, we first remove the technical assumption that $\Omega$ is convex,  and then we provide more qualitative information by digging out the dependence of solution upper bounds  on  the involving model parameters. Lastly,  we extend lower dimensional spaces to  any   dimensional ones.
\item[(ii)] Even though we cannot apply the arguments in \cite{JX-CMS-2021} to obtain large time behaviors of bounded solutions in that fashion, due to especially the qualitative bounds for $\|v\|_{L^\infty}$ and $\|\nabla w\|_{L^\infty}$ in \eqref{v-infty-bdd} and \eqref{gradw-infty-bdd}, upon successfully establishing higher order gradient estimates, we are able to show convergence and exponential convergence of bounded solutions in qualitative  ways.
\end{itemize}
\end{remark}

As a direct consequence of our main results, we obtain unconditional boundedness and convergence to constant steady states of \eqref{system} with $\chi=a=\mu=0$.
\begin{corollary}
All solutions $(u,v,w)$ to the IBVP \eqref{system} with $\chi=0$ are globally bounded. Moreover, if $a=\mu=0$, they converge exponentially to  $(\bar{u}_0,  \bar{u}_0, 0)$ as $t\rightarrow \infty$.
\end{corollary}

It can be easily seen from   Theorem \ref{GB} that the dynamics of \eqref{system} with  small $\xi_1$ or simply $\xi_1=0$ remain largely open. We shall leave this  as a future project.

In the rest of this section, we outline the plan  as well as  main ideas of this article.

In Section 2, we first state the local well-posedness and extensibility criterion  and then derive  some basic properties for \eqref{system}.  Next,  we  extend the  interpolation type inequalities in \cite[Lemma 4.2]{TW-NA-2021} to general setting in Lemma \ref{abs-ineq}.  Finally,   we collect some abstract functionals inequalities including well-known smoothing  $L^p$-$L^q$ estimates of the Neumann heat group in $\Omega$, etc.

In section 3,  we aim to show the proof of qualitative boundedness in Theorem \ref{GB}. Our subtle analysis,  inspired from \cite{TW-NA-2021}, begins with the qualitative control of  $\|v\|_{L^\infty}$ via Moser-iteration in a careful manner, cf. Lemma \ref{b-v}.  Then,  thanks to the generalized interpolation type inequalities in   Lemma \ref{abs-ineq}, under one of the conditions in Theorem \ref{GB} and upon some skillful treatments, we  successfully establish a key ODI (ordinary differential inequality) for the time derivative of the coupled quantity $\|u\|_{L^k}^k+\|\nabla v\|_{L^{2k}}^{2k}$ for some $k>n/2$,  which allows us to derive qualitative bounds for $\|u\|_{L^k}+\|\nabla v\|_{L^{2k}}$, cf. Lemma \ref{high-reg}. Thereafter,  with subtle analysis via semigroup estimates, elliptic estimates and Sobolev embeddings, we finally  conclude the desired uniform-in-time qualitative bounds  as stated in \eqref{GB-1}.

In Section 4, we  proceed further to derive Schauder type estimates of $u$ and $L^{2n}$-estimate of $\nabla u$ so as to  study long time dynamics of bounded solutions to \eqref{system}. Since both the  $u$- and $v$-equation in \eqref{system} have cross-diffusions, the derivation of boundedness of $\|\nabla u\|_{L^{2n}}$ becomes lengthy and technical.  Roughly speaking, motivated from  \cite[Section 3.3]{LX-20-CVPDE}, we first establish an ODI for the time evolution of the coupled quantity $\|u\|_{L^2}^2+\|\Delta v\|_{L^2}^2$ to obtain a  bound for $\|\Delta v\|_{L^2}$. And then, we directly derive an ODI for $\|\nabla u\|_{L^2}^2$ so as to obtain a bound of  $\|\nabla u\|_{L^2}$. This enables us to handle the emerging boundary integral and thus allows us to derive an ODI involving $\|\Delta v\|_{L^{2n}}^{2n}$ for  the time derivative of $\|\nabla u\|_{L^{2n}}^{2n}$. Finally, applying the widely known maximal Sobolev regularity to the $v$-equation in \eqref{system}, we obtain the boundedness of $\|\nabla u\|_{L^{2n}}$, see Lemma \ref{up}.

With qualitative bounds in Section 3 and enhanced regularity properties in Section 4, in Section 5, we aim to study the large time behavior of bounded solutions to \eqref{system} . In the absence of ECs mitosis ($a=\mu=0$), based on the important fact that  upper bounds of  $\|v\|_{L^\infty}$ and $\|\nabla w\|_{L^\infty}$ are  bounded for  $d$ away from zero (cf. \eqref{v-infty-bdd} and \eqref{gradw-infty-bdd}), for $d$ suitably large, we  deduce a  Lyapunov functional for  the coupled quantity
$$
\int_\Omega u\ln \frac{u}{\bar{u}}+\frac{\chi}{2}\int_\Omega |\nabla v|^2,
$$
which enables us to derive the $L^2$- convergence and decay rate of $(u-\bar{u}_0,\nabla v)$ to $(0,0)$. Then using the Gagliardo-Nirenberg inequality together with the enhanced regularity of $(u,v)$ and standard elliptic estimates, we achieve the convergence properties of bounded solution $(u,v, w)$  as in \eqref{conv-main-1} and \eqref{conv-main-2} of Theorem \ref{LT}, see details in Section 5.2.

The convergence analysis for the case $a,\mu>0$ follows in a similar manner. Indeed, for $\mu$ satisfying \eqref{mu-large-lt-intro}, we are able to derive  a  Lyapunov functional for the coupled quantity
$$
\int_\Omega (u-b-b\ln \frac{u}{b})+ \frac{b\chi^2}{2d}\int_\Omega(v-b)^2, \ \ \ b=:(\frac{a}{\mu})^\frac{1}{\theta},
$$
which yields the key starting $L^2$-convergence of $(u-b, v-b)$ to $(0,0)$. Then we can easily use GN interpolation inequality to lift this $L^2$-convergence to $L^\infty$-convergence. Finally, the convergence property of $w$ follows from the  elliptic estimate applied to the $w$-equation in \eqref{system}, thus achieving \eqref{cv-uvw-main3} of Theorem \ref{LT},  see Section 5.2 for details.

\section{Local existence and preliminaries}
In the sequel, the integral $\int_\Omega f(x)dx$  and $\|f\|_{L^p(\Omega)}$ will be abbreviated as $\int_\Omega f$ and $\|f\|_{L^p}$, respectively. The generic constants $c_i$ (defined within the proof of lemmas) or $C_i$ (defined in the statements of lemmas) for $i=1,2,\cdots$, depending  on $n,\Omega$ and the initial data $u_0, v_0$  but they are independent of $t$ and the model parameters   $\chi,\xi_1,\xi_2,d, \mu$,  will vary line-by-line. The existence and uniqueness of local solutions of \eqref{system}  can be easily  shown in a fixed point theorem framework by means of the Amann's theorems \cite{Am-1990,Am-1993} and the parabolic/elliptic regularity theory, as similarly demonstrated in \cite{JX-CMS-2021, Tao-JMAA, TW-NA-2021, W-JMAA-2011}.
\begin{lemma}[Local existence]\label{LS}
    Let  $\Omega\subset\mathbb{R}^n(n\geq 1)$ be a bounded and smooth domain, the model parameters $\chi, \xi_1,\xi_2,d>0, a,\mu,\theta\geq0  $, and,  let   the initial data $(u_0,v_0)$ satisfy  \eqref{initial}. Then there exist a maximal time $T_{max}\in(0,\infty]$ and  a unique triple $(u,v,w)$ of functions with $u$ and $v$ positive  which solves \eqref{system} classically on $\bar{\Omega}\times(0,T_{max})$, and satisfies
\begin{equation*}
\begin{cases}
u\in C^0(\bar{\Omega}\times[0,T_{max}))\cap C^{2,1}(\bar{\Omega}\times(0,T_{max})), \\
v\in \displaystyle\cap_{p>n}C^0([0,T_{max});W^{1,p}(\Omega))\cap C^{2,1}(\bar{\Omega}\times(0,T_{max})),\\
w\in C^{2,0}(\bar{\Omega}\times(0,T_{max})).
\end{cases}
\end{equation*}
Moreover, if $T_{max}<\infty$,  then, for any $p>\max\{n,2\}$,
\begin{equation}\label{KS-1}
\lim\limits_{t\nearrow T_{max}}\sup\left\{\|u(\cdot,t)\|_{L^\infty}+\|v(\cdot,t)\|_{W^{1,p}}\right\}=\infty.
\end{equation}
\end{lemma}

\begin{lemma}[Young's inequality with $\varepsilon$]\label{YI} Let $1<p,q<\infty$, $\frac{1}{p}+\frac{1}{q}=1$. Then
    \begin{equation*}
    XY\leq \varepsilon X^P+\frac{Y^q}{q\left(\varepsilon p\right)^\frac{q}{p}}\ \ (X,Y>0,\ \varepsilon>0).
    \end{equation*}
\end{lemma}

\begin{lemma}
Let $(u,v,w)$ be a solution of \eqref{system} obtained in Lemma \ref{LS}. Then
\begin{equation}\label{b-uvw1}
  |\Omega|\bar{u}=\|u(\cdot,t)\|_{L^1}\leq  m_1,  \ \ \ \forall t\in(0,T_{max}),
\end{equation}
where
\begin{equation*}
m_1:=\begin{cases}
\|u_0\|_{L^1}, &\mathrm{if}\ \  a=\mu=0,\\[0.25cm]
\|u_0\|_{L^1}+(1+a)^{\frac{1+\theta}{\theta}}
    (\frac{1}{\mu})^{\frac{1}{\theta}}(\frac{2}{\theta+1})^{\frac{1}{\theta}}\frac{\theta}{\theta+1}
    |\Omega|, &\mathrm{if}\ \  \mu>0,
\end{cases}
\end{equation*}
and
\begin{equation}\label{b-uvw3}
    \|v(\cdot,t)\|_{L^1}\leq m_1+\|v_0\|_{L^1}.
\end{equation}
%for $\forall~t\in(0,T_{max})$, as well as
%\begin{equation}\label{b-uvw2}
%    \int_{t-\tau}^t\int_\Omega u^{\theta+1}(x,s)dxds\leq {\color{blue}2(1+a)^{\frac{1+\theta}{\theta}}(\frac{1}{\mu})^{\frac{1+\theta}{\theta}}(\frac{2}{\theta+1})^{\frac{1}{\theta}}\frac{\theta}{\theta+1}|\Omega|(1+\tau)+\frac{2\|u_0\|_{L^1}}{\mu}=:m_2},
%\end{equation}
%for $\forall~t\in(\tau,T_{max})$, with $\tau:=\min\{1,\frac{1}{2}T_{max}\}$.
\end{lemma}
\begin{proof}
Integrating the first and second equations of \eqref{system} over $\Omega$ respectively, one has
\begin{equation}\label{b-uvw4*}
    \frac{d}{dt}\int_\Omega u+\mu\int_\Omega u^{\theta+1}=a\int_\Omega u ,
\end{equation}
and
\begin{equation}\label{b-uvw4}
    \frac{d}{dt}\int_\Omega v+\int_\Omega v=\int_\Omega u.
\end{equation}
If $a=\mu=0$, one can easily check that \eqref{b-uvw1} and \eqref{b-uvw3} hold by integrating  \eqref{b-uvw4*} and \eqref{b-uvw4}.

Next,  if $\mu>0$, it follows from the Young's inequality with $\varepsilon$ in Lemma \ref{YI} that
\begin{align*}
(1+a)\int_\Omega u\leq \frac{\mu}{2}\int_\Omega u^{\theta+1}+(1+a)^{\frac{1+\theta}{\theta}}
    (\frac{1}{\mu})^{\frac{1}{\theta}}(\frac{2}{\theta+1})^{\frac{1}{\theta}}\frac{\theta}{\theta+1}
    |\Omega| ,
    \end{align*}
which, upon being substituted into \eqref{b-uvw4*}, gives
\begin{equation}\label{b-uvw5}
    \frac{d}{dt}\int_\Omega u+\int_\Omega u+\frac{\mu}{2}\int_\Omega u^{\theta+1}\leq
  (1+a)^{\frac{1+\theta}{\theta}}
    (\frac{1}{\mu})^{\frac{1}{\theta}}(\frac{2}{\theta+1})^{\frac{1}{\theta}}\frac{\theta}{\theta+1}
    |\Omega|.
\end{equation}
Then    multiplying  \eqref{b-uvw5} by the factor $e^t$  and then integrating, we  simply get
\begin{equation*}
   \int_\Omega u
    \leq (1+a)^{\frac{1+\theta}{\theta}}
    (\frac{1}{\mu})^{\frac{1}{\theta}}(\frac{2}{\theta+1})^{\frac{1}{\theta}}\frac{\theta}{\theta+1}
    |\Omega|+\int_\Omega u_0,
\end{equation*}
which entails \eqref{b-uvw1}. Then substituting \eqref{b-uvw1}  into \eqref{b-uvw4}, one has
\begin{equation}\label{b-uvw6}
    \frac{d}{dt}\int_\Omega v+\int_\Omega v\leq(1+a)^{\frac{1+\theta}{\theta}}(\frac{1}{\mu})^{\frac{1}{\theta}}(\frac{2}{\theta+1})^{\frac{1}{\theta}}\frac{\theta}{\theta+1}|\Omega|+\int_\Omega u_0.
\end{equation}
Solving this ODI or using the Gr\"onwall's inequality again, then \eqref{b-uvw6} implies \eqref{b-uvw3}.
\end{proof}
For our later boundedness purpose, let us  generalize the  interpolation type inequalities in \cite[Lemma 4.2]{TW-NA-2021} to the general case as follows.
\begin{lemma} \label{abs-ineq} Let $\Omega\subset \mathbb{R}^n(n\geq 1)$ be a bounded and smooth domain and let $g, h\in C^2(\bar{\Omega})$ with  $\frac{\partial g}{\partial \nu}|_{\partial\Omega}=\frac{\partial h}{\partial \nu}|_{\partial\Omega}=0$. Then, for all $p\geq 1$,
\begin{equation}\label{inter-1}
    \left|\int_\Omega|\nabla g|^{2p-2}\nabla g\cdot\nabla(\nabla g\cdot\nabla h)\right|\leq\left(\frac{\sqrt{n}}{2p}+1\right)\|\nabla g\|^{2p}_{L^{2(p+1)}}\|D^2h\|_{L^{p+1}},
\end{equation}
and
\begin{equation}\label{inter-2}
\begin{split}
  &\left|\int_\Omega g\Delta h\nabla\cdot(|\nabla g|^{2p-2}\nabla g)\right|\\
  &\leq\left(2(p-1)+\sqrt{n}\right)\|g\|_{L^\infty}\|\nabla g\|^{p-1}_{L^{2(p+1)}}\|\Delta h\|_{L^{p+1}}\left(\int_\Omega|\nabla g|^{2p-2}|D^2g|^2\right)^{\frac{1}{2}},
  \end{split}
\end{equation}
as well as
\begin{equation}\label{inter-3}
    \int_\Omega|\nabla g|^{2(p+1)}\leq\left(2p+\sqrt{n}\right)^2\|g\|^2_{L^\infty}\int_\Omega|\nabla g|^{2(p-1)}|D^2g|^2.
\end{equation}
\end{lemma}
\begin{proof}
Using the following two facts
$$\nabla(\nabla g\cdot\nabla h)=D^2g\cdot\nabla h+D^2h\cdot\nabla g ,$$
and
\begin{equation}\label{e1}
\nabla|\nabla g|^{2p}=p|\nabla g|^{2(p-1)}\nabla|\nabla g|^2=2p|\nabla g|^{2(p-1)}D^2g\cdot\nabla g,
\end{equation}
we thus use  the symmetry of $D^2g$ and integration by parts formula to derive that
\begin{equation}\label{inter-4}
\begin{split}
    &\int_\Omega|\nabla g|^{2p-2}\nabla g\cdot\nabla(\nabla g\cdot\nabla h)\\
    &=\int_\Omega|\nabla g|^{2p-2}\nabla g\cdot\left(D^2g\cdot\nabla h+ D^2h\cdot\nabla g\right)\\
    &=\int_\Omega|\nabla g|^{2p-2}(D^2g\cdot\nabla g)\cdot\nabla h+\int_\Omega|\nabla g|^{2p-2}\nabla g\cdot(D^2h\cdot\nabla g)\\
    &=\frac{1}{2p}\int_\Omega\nabla|\nabla g|^{2p}\cdot\nabla h+\int_\Omega|\nabla g|^{2p-2}\nabla g\cdot(D^2h\cdot\nabla g)\\
    &=-\frac{1}{2p}\int_\Omega|\nabla g|^{2p}\cdot\Delta h+\int_\Omega|\nabla g|^{2p-2}\nabla g\cdot(D^2h\cdot\nabla g).
\end{split}
\end{equation}
Noting the fact  $|\Delta h|\leq\sqrt{n}|D^2h|$ and using H\"older's  inequality, from \eqref{inter-4} we have
\begin{align*}
    \left|\int_\Omega|\nabla g|^{2p-2}\nabla g\cdot\nabla(\nabla g\cdot\nabla h)\right|
    \leq&\frac{\sqrt{n}}{2p}\int_\Omega|\nabla g|^{2p}|D^2h|+\int_\Omega|\nabla g|^{2p}|D^2h|\\
    \leq&(\frac{\sqrt{n}}{2p}+1)\int_\Omega|\nabla g|^{2p}|D^2h|\\
    \leq&(\frac{\sqrt{n}}{2p}+1)\|\nabla g\|^{2p}_{L^{2(p+1)}}\|D^2h\|_{L^{p+1}},
\end{align*}
which is our desired estimate  \eqref{inter-1}. Similarly, we infer that
\begin{align*}
    &\left|\int_\Omega g\Delta h\nabla\cdot(|\nabla g|^{2p-2}\nabla g)\right|\\
    =&\left|\int_\Omega g\Delta h\nabla(|\nabla g|^{2p-2})\cdot\nabla g+\int_\Omega g\Delta h|\nabla g|^{2p-2}\Delta g\right|\\
    =&\left|2(p-1)\int_\Omega g\Delta h|\nabla g|^{2(p-2)}\nabla g\cdot(D^2g\cdot\nabla g)+\int_\Omega g\Delta h|\nabla g|^{2p-2}\Delta g\right|\\
    \leq&(2(p-1)+\sqrt{n})\int_\Omega|g|\cdot|\Delta h|\cdot|\nabla g|^{2p-2}\cdot|D^2g|\\
    \leq &(2(p-1)+\sqrt{n})\|g\|_{L^\infty}\|\Delta h\|_{L^{p+1}}\|\nabla g\|^{p-1}_{L^{2(p+1)}}\||\nabla g|^{p-1}|D^2g|\|_{L^2},
\end{align*}
which gives rise to  \eqref{inter-2}.

Finally, we  use \eqref{e1} to estimate the term  $\int_\Omega|\nabla g|^{2(p+1)}$ as follows:
\begin{align*}
    \int_\Omega|\nabla g|^{2(p+1)}&=\int_\Omega|\nabla g|^{2p}\nabla g\cdot\nabla g\\
    &=-\int_\Omega g|\nabla g|^{2p}\Delta g-\int_\Omega g\nabla(|\nabla g|^{2p})\cdot\nabla g\\
    &=-\int_\Omega g|\nabla g|^{2p}\Delta g-2p\int_\Omega g|\nabla g|^{2(p-1)}\nabla g\cdot(D^2g\cdot\nabla g)\\
    &\leq(2p+\sqrt{n})\int_\Omega|g|\cdot|\nabla g|^{2p}\cdot|D^2g|\\
    &\leq(2p+\sqrt{n})\|g\|_{L^\infty}\left(\int_\Omega|\nabla g|^{2(p+1)}\right)^{\frac{1}{2}}\left(\int_\Omega|\nabla g|^{2(p-1)}|D^2g|^2\right)^{\frac{1}{2}},
\end{align*}
which upon  simple algebraic manipulations  entails \eqref{inter-3}.
\end{proof}
Now, for convenience of reference, we collect the well-known smoothing  $L^p$-$L^q$ estimates of the Neumann heat group in $\Omega$, which can   be found in \cite{Cao15, W-JDE-2010}.
\begin{lemma}\label{SG}
Let $(e^{t\Delta})_{t\geq0}$ be the Neumann heat semigroup in $\Omega$, and let $\lambda_1>0$ denote the first nonzero eigenvalue of $-\Delta$ in $\Omega$ under Neumann boundary conditions. Then there exist some positive constants $c_i~(i=1,2,3)$ depending only on $\Omega$ such that:
\begin{itemize}
 \item[(i)] If $1\leq q\leq p\leq\infty$, then
   \begin{equation}\label{SG-1}
     \|e^{t\Delta}f\|_{L^p}\leq c_1\left(1+t^{-\frac{n}{2}(\frac{1}{q}-\frac{1}{p})}\right)e^{-\lambda_1t}\|f\|_{L^q}~\text{for~all}~t>0
   \end{equation}
   holds for all $f\in L^q(\Omega)$ satisfying $\int_\Omega f=0$.
 \item[(ii)] If $1\leq q\leq p\leq \infty$, then
   \begin{equation*}
     \|\nabla e^{t\Delta}f\|_{L^p}\leq c_2\left(1+t^{-\frac{1}{2}-\frac{n}{2}(\frac{1}{q}-\frac{1}{p})}\right)e^{-\lambda_1t}\|f\|_{L^q}~\text{for~all}~t>0
   \end{equation*}
   is valid for all $f\in L^q(\Omega)$.
 \item[(iii)] If $2\leq q\leq p< \infty$, then
   \begin{equation*}
     \|\nabla e^{t\Delta}f\|_{L^p}\leq c_3\left(1+t^{-\frac{n}{2}(\frac{1}{q}-\frac{1}{p})}\right)e^{-\lambda_1t}\|\nabla f\|_{L^q}~\text{for~all}~t>0
   \end{equation*}
   is true for all $f\in W^{1,p}(\Omega)$.
 \item[(iv)] If $1< q\leq p\leq \infty$, then
   \begin{equation}\label{SG-4}
     \|e^{t\Delta}\nabla\cdot f\|_{L^p}\leq c_4\left(1+t^{-\frac{1}{2}-\frac{n}{2}(\frac{1}{q}-\frac{1}{p})}\right)e^{-\lambda_1t}\|f\|_{L^q}~\text{for~all}~t>0
   \end{equation}
 is valid for all $f\in (W^{1,p}(\Omega))^n$.
\end{itemize}
\end{lemma}

\begin{lemma}[\cite{JW-preprint}]\label{II}
Let $f(x,t)$ be  a positive function for $(x,t)\in\Omega \times (0,\infty)$ and define  $\bar{f}=\frac{1}{|\Omega|}\int_\Omega f$. Then it holds that
\begin{equation*}
0\leq \frac{1}{2\bar{f}}\|f-\bar{f}\|_{L^1}^2\leq \int_\Omega f\ln \frac{f}{\bar{f}}\leq \frac{1}{\bar{f}}\|f-\bar{f}\|_{L^2}^2.
\end{equation*}
\end{lemma}

\section{Qualitative boundedness: Proof of Theorem \ref{GB}}
In this section, we are devoted to proving the qualitative boundedness in terms of model parameters as stated in Theorem \ref{GB}. To this end, we first use the convection effect to show the boundedness of $\|v\|_{L^\infty}$ based on some ideas in \cite{TW-NA-2021}.
\begin{lemma}\label{b-v}
Let $(u,v,w)$ be the solution of \eqref{system} obtained  from  Lemma \ref{LS}. Then
\begin{equation}\label{b-v1}
\begin{split}
   &\|v(\cdot,t)\|_{L^\infty}\\
   &\leq C_0\max\left\{m_1+\|v_0\|_{L^1},\frac{1}{\xi_2},  \|v_0\|_{L^\infty}\right\}\left(1
+m_1\xi_2+(\frac{1}{d})^{\frac{n}{2}}(m_1\xi_2)^{1+\frac{n}{2}}\right):=M_0,
\end{split}
\end{equation}
for all $t\in[0, T_{\text{max}})$; here $C_0>0$ is defined in \eqref{c0-formula}, which  depends only on $n$ and $\Omega$.
\end{lemma}

\begin{proof}
For $p>1$, multiplying the second equation of \eqref{system} by $v^{p-1}$, integrating the result over $\Omega$ by parts and combining the equation $\Delta w=-u+\bar{u}$, we obtain that
\begin{align*}
    \frac{1}{p}\frac{d}{dt}\int_\Omega v^p&=d\int_\Omega\Delta v\cdot v^{p-1}+\xi_2\int_\Omega \nabla\cdot(v\nabla w)v^{p-1}+\int_\Omega uv^{p-1}-\int_\Omega v^p\\
    &=-(p-1)d\int_\Omega v^{p-2}|\nabla v|^2-\xi_2(p-1)\int_\Omega v^{p-1}\nabla v\cdot\nabla w+\int_\Omega uv^{p-1}-\int_\Omega v^p\\
    &=-(p-1)d\int_\Omega v^{p-2}|\nabla v|^2+\frac{\xi_2(p -1)}{p}\int_\Omega v^p\Delta w+\int_\Omega uv^{p-1}-\int_\Omega v^p\\
    &=-(p-1)d\int_\Omega v^{p-2}|\nabla v|^2+\frac{\xi_2(p-1)}{p}\int_\Omega v^p(-u+\bar{u})+\int_\Omega uv^{p-1}-\int_\Omega v^p,
\end{align*}
which upon the fact $v^{p-2}|\nabla v|^2=\frac{4}{p^2}|\nabla v^{\frac{p}{2}}|^2$ gives
\begin{equation}\label{b-v2}
\begin{split}
    &\frac{d}{dt}\int_\Omega v^p+\frac{4d(p-1)}{p}\int_\Omega|\nabla v^{\frac{p}{2}}|^2+\xi_2(p-1)\int_\Omega uv^p+p\int_\Omega v^p\\
   & =\xi_2(p-1)\bar{u}\int_\Omega v^p+p\int_\Omega uv^{p-1}.
    \end{split}
\end{equation}
Applying the Young's inequality with $\varepsilon$ (cf. Lemma \ref{YI}) and  the facts $\|u(\cdot,t)\|_{L^1}\leq m_1$ and $\bar{u}=\frac{1}{|\Omega|}\int_\Omega u\leq \frac{m_1}{|\Omega|}$, we infer
\begin{equation}\label{b-v3}
    \begin{split}
       p \int_\Omega uv^{p-1}
        &\leq \xi_2(p-1)\int_\Omega uv^p+\xi_2^{-(p-1)}\int_\Omega u\\
        &\leq \xi_2(p-1)\int_\Omega uv^p+m_1\xi_2^{1-p},
    \end{split}
\end{equation}
and
\begin{equation}\label{b-v4}
\xi_2(p-1)\bar{u}\int_\Omega  v^p\leq \frac{\xi_2m_1(p-1)}{|\Omega|}\int_\Omega v^p.
\end{equation}
A substitution of  \eqref{b-v3} and \eqref{b-v4} into \eqref{b-v2} shows that
\begin{equation}\label{b-v5}
\begin{split}
    &\frac{d}{dt}\int_\Omega v^p+\frac{4d(p-1)}{p}\int_\Omega|\nabla v^{\frac{p}{2}}|^2+p\int_\Omega v^p\leq \frac{\xi_2m_1(p-1)}{|\Omega|}\int_\Omega v^p+m_1\xi_2^{1-p}.
    \end{split}
\end{equation}
For $\varepsilon>0$, there exists $c_1>0$ depending only on $n$ and $\Omega$ such that (cf. \cite{TW-M3AS-2013, Xiangjde})
\begin{equation}\label{b-v5*}
\|U\|_{L^2}^2\leq \varepsilon\|\nabla U\|_{L^2}^2+c_1(1+\varepsilon^{-\frac{n}{2}})\|U\|_{L^1}^2.
\end{equation}
 We choose  $U=v^{\frac{p}{2}}$ and $\varepsilon=\frac{4d|\Omega|}{p\xi_2 m_1}$ in \eqref{b-v5*} to derive that
\begin{equation}\label{b-v6}
\begin{split}
 &\frac{\xi_2m_1(p-1)}{|\Omega|}\int_\Omega v^p\\
 &\leq \frac{4(p-1)d}{p}\int_\Omega|\nabla v^{\frac{p}{2}}|^2+\frac{\xi_2m_1 c_1(p-1)}{|\Omega|}\left(1+\left(\frac{\xi_2m_1p}{4d|\Omega|}
 \right)^\frac{n}{2}\right)\left(\int_\Omega v^{\frac{p}{2}}\right)^2\\
 &\leq \frac{4(p-1)d}{p}\int_\Omega|\nabla v^{\frac{p}{2}}|^2+c_2p(1+p^\frac{n}{2})\left(\int_\Omega v^{\frac{p}{2}}\right)^2,
 \end{split}
\end{equation}
where $c_2:=\frac{\xi_2m_1 c_1}{|\Omega|}\max\{1,\left(\frac{\xi_2m_1}{4d|\Omega|}
 \right)^\frac{n}{2}\}$. Then substituting \eqref{b-v6} into \eqref{b-v5} and noting the fact $(1+p^\frac{n}{2})\leq 2(1+p)^\frac{n}{2}$, we obtain
\begin{equation*}
\begin{split}
    &\frac{d}{dt}\int_\Omega v^p+p\int_\Omega v^p\leq 2c_2p(1+p)^\frac{n}{2}\left(\int_\Omega v^{\frac{p}{2}}\right)^2+m_1\xi_2^{1-p},
    \end{split}
\end{equation*}
which implies that
\begin{equation}\label{b-v7}
\begin{split}
 \frac{d}{dt}\left(e^{pt}\int_\Omega v^p\right)\leq 2e^{pt}c_2p(1+p)^\frac{n}{2}\left(\int_\Omega v^{\frac{p}{2}}\right)^2+ e^{pt}m_1\xi_2^{1-p}.
 \end{split}
\end{equation}
For any $T\in(0, T_{\text{max}})$, we integrate \eqref{b-v7} over  $[0,t]$ for $0\leq t\leq T$ to obtain
\begin{equation*}
    \begin{split}
        \int_\Omega v^p(x,t)
        &\leq\int_\Omega v_0^p+2c_2(1+p)^\frac{n}{2}\sup_{0\leq t\leq T}\left(\int_\Omega v^{\frac{p}{2}}(x,t)\right)^2+\frac{m_1\xi_2^{1-p}}{p},
    \end{split}
\end{equation*}
{\color{black}
which immediately yields
\begin{align*}
    &\ \ \ \ \left(\int_\Omega v^p(x,t) \right)^\frac{1}{p} \nonumber\\
    &\leq \left[2c_2(1+p)^{\frac{n}{2}}\sup_{0\leq t\leq T}\left(\int_\Omega v^{\frac{p}{2}}(x,t)\right)^2+m_1\xi_2^{1-p}+|\Omega|\|v_0\|_{L^\infty}^p\right]^\frac{1}{p}  \nonumber\\
    &\leq \left(\max\left\{2c_2,m_1\xi_2,|\Omega|\right\}\right)^\frac{1}{p}(1+p)^{\frac{n}{2p}}\left\{\sup_{0\leq t\leq T}\left(\int_\Omega v^{\frac{p}{2}}(x,t)\right)^\frac{2}{p}+\frac{1}{\xi_2}+\|v_0\|_{L^\infty}\right\}.
\end{align*}
Therefore,  it follows with $c_3:=3\max\{2c_2,m_1\xi_2,|\Omega|, 1\}$ that
\begin{equation}\label{b-v9***}
   \begin{split}
    &\ \ \ \ \max\left\{\sup_{0\leq t\leq T}\left(\int_\Omega v^p(x,t)\right)^\frac{1}{p}, \ \ \frac{1}{\xi_2}, \ \ \|v_0\|_{L^\infty}\right\}\\
    &\leq c_3^\frac{1}{p}(1+p)^{\frac{n}{2p}}\max\left\{\sup_{0\leq t\leq T}\left(\int_\Omega v^{\frac{p}{2}}(x,t)\right)^\frac{2}{p}, \ \ \frac{1}{\xi_2}, \ \ \|v_0\|_{L^\infty}\right\}.
   \end{split}
\end{equation}
Upon setting
\[
H(p):=\max\left\{\sup_{0\leq t\leq T}\left(\int_\Omega v^p(x,t)\right)^{\frac{1}{p}}, \ \ \frac{1}{\xi_2}, \ \ \|v_0\|_{L^\infty}\right\},
\]
then  \eqref{b-v9***} becomes
\[
    H(p)\leq c_3^{\frac{1}{p}}(1+p)^{\frac{n}{2p}}H(\frac{p}{2}), \ \ \ \forall p\geq 2.
\]
Taking $p=2^j~(j=1,2,\cdots)$, we obtain inductively that
\begin{equation}
\begin{split}\label{b-v11}
H(2^j)&\leq c_3^{2^{-j}}(1+2^j)^{\frac{n}{2}\cdot 2^{-j}}H(2^{j-1}) \\
&\leq c_3^{2^{-j}+2^{-(j-1)}}(1+2^j)^{\frac{n}{2}\cdot 2^{-j}}(1+2^{j-1})^{\frac{n}{2}\cdot 2^{-(j-1)}}H(2^{j-2})   \\
&\leq c_3^{\sum\limits_{i=0}^22^{-(j-i)}}\cdot\prod_{k=j-2}^j(1+2^k)^{\frac{1}{2^k}\cdot \frac{n}{2}} H(2^{j-3})   \\
&\cdots\cdots  \\
&\leq c_3^{\sum\limits_{i=0}^{j-1}2^{-(j-i)}}\cdot\prod_{k=1}^j(1+2^k)^{\frac{1}{2^k}\cdot \frac{n}{2}}H(1)   \\
&=c_3^{1-\frac{1}{2^j}}\prod_{k=1}^j(1+2^k)^{\frac{1}{2^k}\cdot \frac{n}{2}}H(1).
\end{split}
\end{equation}
 On the one hand, using the fact that $\ln(1+z)\leq \sqrt{z}$ for all $z\geq0$, we have
$$
\ln\prod_{k=1}^j(1+2^k)^{\frac{1}{2^k}}=\sum_{k=1}^j\frac{\ln(1+2^k)}{2^k}\leq \sum_{k=1}^\infty\left(\frac{1}{\sqrt{2}}\right)^k\leq 6,\ \forall j=1,2,\cdots,
$$
and so,
\begin{equation}\label{b-v12}
    \lim_{j\to\infty}\prod_{k=1}^j(1+2^k)^{\frac{1}{2^k}\cdot \frac{n}{2}}\leq e^{3n}.
\end{equation}
 Then the combination of \eqref{b-v11}, \eqref{b-v12} and \eqref{b-uvw3} with $c_2:=\frac{\xi_2m_1c_1}{|\Omega|}\max\{1,(\frac{\xi_2m_1}{4d|\Omega|})^{\frac{n}{2}}\}$ gives
\begin{equation}\label{v-linfty-bdd}
\begin{split}
\|v(\cdot,t)\|_{L^\infty}&\leq \lim_{j\to\infty} H(2^j)\\
&\leq c_3e^{3n}H(1)\\
&=3\max\{2c_2,m_1\xi_2,|\Omega|,1\}e^{3n}\max\left\{\sup_{0\leq t\leq T}\int_\Omega v(\cdot,t), \ \frac{1}{\xi_2}, \ \|v_0\|_{L^\infty}\right\}\\
&\leq 3e^{3n}\left(1+|\Omega|+m_1\xi_2+2c_2\right)\max\left\{ m_1+\ \|v_0\|_{L^1},\frac{1}{\xi_2}, \ \|v_0\|_{L^\infty}\right\}\\
&\leq C_0\max\left\{  m_1+\|v_0\|_{L^1}, \frac{1}{\xi_2},  \|v_0\|_{L^\infty}\right\}\left((1
+m_1\xi_2+(\frac{1}{d})^{\frac{n}{2}}(m_1\xi_2)^{1+\frac{n}{2}}\right)
\end{split}
\end{equation}
for all $t\in[0,T]$, where \begin{equation}\label{c0-formula}
C_0:=3e^{3n}\max\left\{1+|\Omega|,1+\frac{2c_1}{|\Omega|},
\frac{2c_1}{(4|\Omega|)^{\frac{n}{2}}|\Omega|}\right\}.
\end{equation}
Since $T\in(0, T_{\text{max}})$ is arbitrary and the upper bound in \eqref{v-linfty-bdd} is independent of $T$, the desired estimate \eqref{b-v1} follows from \eqref{v-linfty-bdd} and \eqref{c0-formula}.}
\end{proof}

\subsection{Qualitative $L^p$-estimates} In this subsection, by means of the key estimate on $\|v\|_{L^\infty}$ in \eqref{v-linfty-bdd}, we shall establish further coupled $L^p$-energy estimates with dependence on key parameters. Here, we shall stress that our arguments are applicable to both cases: $a,\mu>0$ and $a=\mu=0$.
\begin{lemma}\label{cee}
 For $p>1$, the local-in-time solution $(u,v,w)$  of \eqref{system} obtained in Lemma \ref{LS} satisfies
\begin{equation}\label{cee-1}
\begin{split}
    &\frac{d}{dt}\int_\Omega u^p+\frac{p(p-1)}{2}\int_\Omega u^{p-2}|\nabla u|^2+\frac{\xi_1(p-1)}{2}
    \int_\Omega u^{p+1}+\frac{p\mu}{2} \int_\Omega u^{\theta+p} \\
    &\leq \frac{\chi^2p(p-1)}{2}\int_\Omega u^p|\nabla v|^2+M_1(p),\ \ \mathrm{for} \ \ t\in(0,T_{max}),
\end{split}
\end{equation}
where
\begin{equation}\label{M1}
M_1(p):=
\begin{cases}
\xi_1m_1^{p+1}\frac{(p-1)}{p+1}\cdot
\left(\frac{2p}{|\Omega|(p+1)}\right)^p+\left(\frac{1}{\mu}\right)^{\frac{p}{\theta}}\left(\frac{2ap}{p+\theta}\right)^{\frac{p}{\theta}}\frac{ap\theta}{p+\theta}|\Omega|, & {\color{black} a, \mu>0,  }\\[0.25cm]
\xi_1m_1^{p+1}\frac{(p-1)}{p+1}\cdot
\left(\frac{2p}{|\Omega|(p+1)}\right)^p, &  a=\mu=0.
\end{cases}
\end{equation}
\end{lemma}
\begin{proof}
For $p>1$, multiplying the first equation of \eqref{system} by $u^{p-1}$ and  integrating the result over $\Omega$ by parts, we conclude that
\begin{align*}
        &\frac{1}{p}\frac{d}{dt}\int_\Omega u^p+(p-1)\int_\Omega u^{p-2}|\nabla u|^2+\mu\int_\Omega u^{\theta+p}\\
        &=\chi(p-1)\int_\Omega u^{p-1}\nabla u\cdot\nabla v-\xi_1(p-1)\int_\Omega u^{p-1}\nabla u\cdot\nabla w+a\int_\Omega u^p\\
        &=\chi(p-1)\int_\Omega u^{p-1}\nabla u\cdot\nabla v+\frac{\xi_1(p-1)}{p}\int_\Omega  u^p\Delta w+a\int_\Omega u^p,
\end{align*}
which together with the fact $\Delta w=\bar{u}-u$ gives
\begin{equation}\label{cee-2}
    \begin{split}
        &\frac{1}{p}\frac{d}{dt}\int_\Omega u^p+(p-1)\int_\Omega u^{p-2}|\nabla u|^2+\mu\int_\Omega u^{\theta+p}+\frac{\xi_1(p-1)}{p}\int_\Omega  u^{p+1}\\
        =~&\chi(p-1)\int_\Omega u^{p-1}\nabla u\cdot\nabla v+\frac{\xi_1(p-1)}{p}\bar{u}\int_\Omega  u^p+a\int_\Omega u^p.
    \end{split}
\end{equation}
Applications of  H\"{o}lder's inequality and Young's inequality yield that
\begin{equation}\label{cee-3}
    \begin{split}
       \chi(p-1)\int_\Omega u^{p-1}\nabla u\cdot\nabla v
       &\leq\frac{(p-1)}{2}\int_\Omega u^{p-2}|\nabla u|^2+\frac{\chi^2(p-1)}{2}\int_\Omega u^p|\nabla v|^2,
    \end{split}
\end{equation}
and noting from  $\bar{u}=\frac{1}{|\Omega|}\int_\Omega u\leq \frac{m_1}{|\Omega|}$ that
\begin{equation}\label{cee-4}
    \begin{split}
        \frac{\xi_1(p-1)}{p}\bar{u}\int_\Omega u^p        &\leq\frac{\xi_1(p-1)m_1}{p|\Omega|}\int_\Omega u^{p}\\
        &\leq\frac{\xi_1(p-1)}{2p}\int_\Omega u^{p+1}+{\color{black}\xi_1m_1^{p+1}\frac{(p-1)}{p(p+1)}\left(\frac{2p}{|\Omega|(p+1)}\right)^p,  }
    \end{split}
\end{equation}
 as well as, in the case of $a, \mu>0$,
\begin{equation}\label{cee-5}
    \begin{split}
         a\int_\Omega u^p
          &\leq
          \frac{\mu}{2}\int_\Omega u^{p+\theta}
          +\left(\frac{2ap}{\mu(p+\theta)}\right)^{\frac{p}{\theta}}
          \frac{{\color{black}a}\theta}{p+\theta}|\Omega|.\\
    \end{split}
\end{equation}
Substituting \eqref{cee-3}, \eqref{cee-4} and \eqref{cee-5} into \eqref{cee-2} and noting the definition of $M_1(p)$ in \eqref{M1}, we end up with  \eqref{cee-1} directly, thus proving this lemma.
\end{proof}

\begin{lemma}\label{reg-v*}
 For $k\geq1$, the local-in-time classical  solution of \eqref{system} obtained in Lemma \ref{LS} satisfies, for $t\in[0, T_{\text{max}})$, that
\begin{equation}\label{Key1}
    \begin{split}
        &\frac{d}{dt}\int_\Omega|\nabla v|^{2k}+2k\int_\Omega|\nabla v|^{2k}+2kd\int_\Omega|\nabla v|^{2k-2}|D^2v|^2\\
        &\ \ \ +k(k-1)d\int_\Omega|\nabla v|^{2k-4}|\nabla|\nabla v|^2|^2\\
        &\leq 2k\xi_2\int_\Omega|\nabla v|^{2k-2}\nabla v\cdot\nabla(\nabla v\cdot\nabla w)-2k\xi_2\int_\Omega v\Delta w\nabla\cdot(|\nabla v|^{2k-2}\nabla v)\\
        &\ \ \ \ +2k(2(k-1)+\sqrt{n})\int_\Omega u|\nabla v|^{2k-2}|D^2v|+dk\int_{\partial\Omega}|\nabla v|^{2k-2}\frac{\partial|\nabla v|^2}{\partial\nu}.
    \end{split}
\end{equation}
\end{lemma}

\begin{proof}
For $k\geq 1$, differentiating the second equation of \eqref{system} and multiplying the result by $|\nabla v|^{2k-2}\nabla v$, we obtain via integration by parts that
\begin{align}\label{reg-v2}
        \frac{1}{2k}\frac{d}{dt}\int_\Omega|\nabla v|^{2k}&=\int_\Omega|\nabla v|^{2k-2}\nabla v\cdot\nabla(d\Delta v+\xi_2\nabla\cdot(v\nabla w)-v+u)\nonumber\\
        &=d\int_\Omega|\nabla v|^{2k-2}\nabla v\cdot\nabla\Delta v+\xi_2\int_\Omega|\nabla v|^{2k-2}\nabla v\cdot\nabla(\nabla\cdot(v\nabla w))\\
        &\;\;\;+\int_\Omega|\nabla v|^{2k-2}\nabla v\cdot\nabla u-\int_\Omega|\nabla v|^{2k} \nonumber\\
        &=:I_1+I_2+I_3-\int_\Omega|\nabla v|^{2k}.\nonumber
\end{align}
Using the identity $\nabla v\cdot\nabla\Delta v=\frac{1}{2}\Delta|\nabla v|^2-|D^2v|^2$ and integrating by parts, we compute
\begin{equation}\label{reg-v3}
\begin{split}
        I_1&=\frac{d}{2}\int_\Omega|\nabla v|^{2k-2}\Delta|\nabla v|^2-d\int_\Omega|\nabla v|^{2k-2}|D^2v|^2\\
        &=\frac{d}{2}\int_{\partial\Omega}|\nabla v|^{2k-2}\frac{\partial|\nabla v|^2}{\partial\nu}-\frac{(k-1)d}{2}\int_\Omega|\nabla v|^{2k-4}|\nabla|\nabla v|^2|^2\\
        & \ \ \ -d\int_\Omega|\nabla v|^{2k-2}|D^2v|^2.
\end{split}
\end{equation}
Similarly, we use integration by parts to rewrite $I_2$ as follows:
\begin{align}\label{reg-v4}
        I_2&=\xi_2\int_\Omega|\nabla v|^{2k-2}\nabla v\cdot\nabla(\nabla v\cdot\nabla w+v\Delta w)\\
        &=\xi_2\int_\Omega|\nabla v|^{2k-2}\nabla v\cdot\nabla(\nabla v\cdot\nabla w)+\xi_2\int_\Omega|\nabla v|^{2k-2}\nabla v\cdot\nabla(v\Delta w)\nonumber\\
        &=\xi_2\int_\Omega|\nabla v|^{2k-2}\nabla v\cdot\nabla(\nabla v\cdot\nabla w)-\xi_2\int_\Omega v\Delta w\nabla\cdot(|\nabla v|^{2k-2}\nabla v). \nonumber
\end{align}
As for $I_3$, using  the fact $|\Delta v|\leq\sqrt{n}|D^2v|$ and the identity \eqref{e1}, we have
\begin{align}\label{reg-v5}
        I_3&=\int_\Omega|\nabla v|^{2k-2}\nabla v\cdot\nabla u\nonumber\\
        &=-\int_\Omega u\nabla|\nabla v|^{2k-2}\cdot\nabla v-\int_\Omega u|\nabla v|^{2k-2}\Delta v\\
        &=-2(k-1)\int_\Omega u|\nabla v|^{2k-4}\nabla v\cdot (D^2v\cdot\nabla v)-\int_\Omega u|\nabla v|^{2k-2}\Delta v\nonumber\\
        &\leq (2(k-1)+\sqrt{n})\int_\Omega u|\nabla v|^{2k-2}|D^2v|. \nonumber
\end{align}
A substitution of  \eqref{reg-v3}, \eqref{reg-v4} and \eqref{reg-v5} into \eqref{reg-v2} shows  \eqref{Key1}.
\end{proof}
Next, we use the $L^\infty$-bound of $v$ provided by Lemma \ref{b-v} to control the terms on the right-hand side of \eqref{Key1}.
\begin{lemma}\label{ve}For $k\geq1$, the local-in-time classical  solution of \eqref{system} obtained in Lemma \ref{LS} satisfies, for $t\in[0, T_{\text{max}})$, that
\begin{equation}\label{ve-1}
    \begin{split}
        &\frac{d}{dt}\int_\Omega|\nabla v|^{2k}+2k\int_\Omega|\nabla v|^{2k}+\frac{k}{2}(3d-d_\Omega)\int_\Omega|\nabla v|^{2k-2}|D^2v|^2\\
        &\ \ \ +k(k-1)(d-d_\Omega)\int_\Omega|\nabla v|^{2k-4}|\nabla|\nabla v|^2|^2\\
        &\leq  \frac{C_1}{d^k}(1+M_0\xi_2)^{k+1}M_0^{k-1}\int_\Omega u^{k+1} +C_1d_\Omega M_0^{2k},
    \end{split}
\end{equation}
where   $C_1>0$  depends only on $n,k, u_0, v_0$ and $\Omega$ and $d_\Omega$ is defined by \eqref{d-omega}.
\end{lemma}

\begin{proof}
By  \eqref{inter-1} and \eqref{inter-3} and  the fact $\|v(\cdot,t)\|_{L^\infty}\leq M_0$ in \eqref{b-v1},  one obtains that
\begin{align}\label{int-coro-2}
        J_1:&=2k\xi_2\int_\Omega|\nabla v|^{2k-2}\nabla v\cdot\nabla(\nabla v\cdot\nabla w) \nonumber\\
        &\leq2k \xi_2(\frac{\sqrt{n}}{2k}+1)\|\nabla v\|^{2k}_{L^{2(k+1)}}\|D^2w\|_{L^{k+1}} \nonumber\\
        &=2k\xi_2(\frac{\sqrt{n}}{2k}+1)\left(\int_\Omega|\nabla v|^{2(k+1)}\right)^{\frac{k}{k+1}}\|D^2w\|_{L^{k+1}}\\
        &\leq2k\xi_2(\frac{\sqrt{n}}{2k}+1)(2k+\sqrt{n})^{\frac{2k}{k+1}}\|v\|^{\frac{2k}{k+1}}_{L^\infty}\left(\int_\Omega|\nabla v|^{2(k-1)}|D^2v|^2\right)^{\frac{k}{k+1}}\|D^2w\|_{L^{k+1}} \nonumber\\
        &\leq2k\xi_2(\frac{\sqrt{n}}{2k}+1)(2k+\sqrt{n})^{\frac{2k}{k+1}}M_0^{\frac{2k}{k+1}}\left(\int_\Omega|\nabla v|^{2(k-1)}|D^2v|^2\right)^{\frac{k}{k+1}}\|D^2w\|_{L^{k+1}} \nonumber
\end{align}
for all $t\in(0,T_{max})$. On the other hand, by the uniqueness of the elliptic problem
\begin{equation}\label{w-2p}
-\Delta w=u-\bar{u} \text{ in } \Omega,  \ \ \   \ \frac{\partial w}{\partial \nu}=0 \text{ on } \partial\Omega, \ \ \   \ \int_\Omega w=0,
\end{equation}
the well-known  $W^{2,p}$-elliptic estimate (cf. \cite{ADN59,ADN64,La}, indeed, $\Delta^{-1}:L^p\rightarrow W^{2,p}$ is a homeomorphism) shows that
\begin{equation}\label{int-coro-3}
\|\Delta w\|_{L^{k+1}}\leq\sqrt{n}\|D^2w\|_{L^{k+1}}\leq c_4\sqrt{n}\|u-\bar{u}\|_{L^{k+1}}\leq 2c_4\sqrt{n}\|u\|_{L^{k+1}}:=c_5\|u\|_{L^{k+1}},
\end{equation}
where  $c_5$ depends only on $n$ and $\Omega$.  Then we  substitute \eqref{int-coro-3} into \eqref{int-coro-2} to get
\begin{equation}\label{int-coro-4}
 J_1\leq2c_5k\xi_2(\frac{1}{2k}+\frac{1}{\sqrt{n}})(2k+\sqrt{n})^{\frac{2k}{k+1}}
{M_0}^{\frac{2k}{k+1}}\left(\int_\Omega|\nabla v|^{2(k-1)}|D^2v|^2\right)^{\frac{k}{k+1}}\|u\|_{L^{k+1}}.
\end{equation}
Furthermore, using \eqref{inter-2}, \eqref{inter-3} and the fact \eqref{int-coro-3}, one can derive
\begin{equation}\label{int-coro-5}
\begin{split}
J_2&:=-2k\xi_2\int_\Omega v\Delta w\nabla\cdot(|\nabla v|^{2k-2}\nabla v)\\
&\leq2k\xi_2(2(k-1)+\sqrt{n})\|v\|_{L^\infty}\|\nabla v\|_{L^{2(k+1)}}^{k-1}\left(\int_\Omega|\nabla v|^{2k-2}|D^2v|^2\right)^{\frac{1}{2}}\|\Delta w\|_{L^{k+1}}\\
&\leq2k\xi_2(2(k-1)+\sqrt{n})(2k+\sqrt{n})^{\frac{k-1}{k+1}}\|v\|_{L^\infty}^{\frac{2k}{k+1}}\left(\int_\Omega|\nabla v|^{2k-2}|D^2v|^2\right)^{\frac{k}{k+1}}\|\Delta w\|_{L^{k+1}}\\
&\leq 2c_5k\xi_2(2(k-1)+\sqrt{n})(2k+\sqrt{n})^{\frac{k-1}{k+1}}
        {M_0}^{\frac{2k}{k+1}}
\left(\int_\Omega|\nabla v|^{2k-2}|D^2v|^2\right)^{\frac{k}{k+1}}\|u\|_{L^{k+1}}.
\end{split}
\end{equation}
In addition, applying the H\"older's inequality together with \eqref{inter-3}, we have
\begin{equation} \label{int-coro-6}
\begin{split}
  J_3:&=2k(2(k-1)+\sqrt{n})\int_\Omega u|\nabla v|^{2k-2}|D^2v|\\
    &\leq2k(2(k-1)+\sqrt{n})\left(\int_\Omega u^2|\nabla v|^{2(k-1)}\right)^{\frac{1}{2}}\left(\int_\Omega|\nabla v|^{2(k-1)}|D^2v|^2\right)^{\frac{1}{2}}\\
    &\leq 2k(2(k-1)+\sqrt{n})\|u\|_{L^{k+1}}\|\nabla v\|_{L^{2(k+1)}}^{k-1}\left(\int_\Omega|\nabla v|^{2(k-1)}|D^2v|^2\right)^{\frac{1}{2}}\\
    &\leq2k(2(k-1)+\sqrt{n})(2k+\sqrt{n})^{\frac{k-1}{k+1}}
 {M_0}^{\frac{k-1}{k+1}}\left(\int_\Omega|\nabla v|^{2(k-1)}|D^2v|^2\right)^{\frac{k}{k+1}}
    \cdot\|u\|_{L^{k+1}}.
\end{split}
\end{equation}
Now, summing over \eqref{int-coro-4}, \eqref{int-coro-5} and \eqref{int-coro-6}, we infer
\begin{equation}\label{int-coro-7}
    J_1+J_2+J_3\leq c_6\left(\int_\Omega|\nabla v|^{2(k-1)}|D^2v|^2\right)^{\frac{k}{k+1}}\|u\|_{L^{k+1}},
\end{equation}
where
\begin{equation} \label{c3-def}
\begin{split}
     c_6&=2c_5k\xi_2(\frac{1}{2k}+\frac{1}{\sqrt{n}})(2k+\sqrt{n})^{\frac{2k}{k+1}}
{M_0}^{\frac{2k}{k+1}}\\
&\ \ +2c_5k\xi_2(2(k-1)+\sqrt{n})(2k+\sqrt{n})^{\frac{k-1}{k+1}}
        {M_0}^{\frac{2k}{k+1}}\\
        &\ \ +2k(2(k-1)+\sqrt{n})(2k+\sqrt{n})^{\frac{k-1}{k+1}}
 {M_0}^{\frac{k-1}{k+1}}\\
 &\leq 6c_5k\xi_2(2k+\sqrt{n})^{\frac{2k}{k+1}}M_0^\frac{2k}{k+1}
 +2k(2k+\sqrt{n})^{\frac{2k}{k+1}}M_0^\frac{k-1}{k+1}\\
 &=2k (1+3c_5\xi_2M_0)(2k+\sqrt{n})^{\frac{2k}{k+1}}M_0^{\frac{k-1}{k+1}}.
\end{split}
\end{equation}
Therefore, by means of the Young's inequality, we infer from  \eqref{int-coro-7} and  \eqref{c3-def} that
\begin{equation}\label{int-coro-7+}
\begin{split}
    &J_1+J_2+J_3\\
    &\leq \frac{k}{2}d\int_\Omega|\nabla v|^{2(k-1)}|D^2v|^2+\frac{c_6^{k+1}}{k+1}\left(\frac{2}{d(k+1)}\right)^k\int_\Omega u^{k+1}\\
    &\leq \frac{k}{2}d\int_\Omega|\nabla v|^{2(k-1)}|D^2v|^2
    +\frac{2^{2k+1}}{d^k}(1+3c_5\xi_2M_0)^{k+1}(2k+\sqrt{n})^{2k}M_0^{k-1}\int_\Omega u^{k+1}.
    \end{split}
\end{equation}
To control the boundary integral on \eqref{Key1}, we first state  the following well-known fact  due to homogeneous Neumann conditions:
\begin{equation}\label{pt-convex}
\frac{\partial|\nabla v|^2}{\partial\nu}\leq 2\sigma_\Omega |\nabla v|^2 \text{ on } \partial\Omega,
\end{equation}
where $\sigma_\Omega=\sigma 1_\Omega$ with $\sigma$ being the maximum curvature of $\partial\Omega$ and $1_\Omega$ being the indicator whether $\Omega$ is non-convex  defined by
$$
1_\Omega=\begin{cases} 0, & \text{ if  } \Omega \text{ is convex \cite[Lemma 5.3]{Ma-79}},\\
1, & \text{ if } \Omega \text{ is non-convex \cite[Lemma 4.2]{MS-Poincare-2014}}.
\end{cases}
$$
Now, combining \eqref{pt-convex} and using  the trace inequality $\|\psi\|_{L^2(\partial\Omega)}\leq\varepsilon\|\nabla\psi\|_{L^2(\Omega)}
+C_\varepsilon\|\psi\|_{L^2(\Omega)}$ for any $\varepsilon>0$ (cf. \cite[Remark 52.9]{SQ07} and \cite[(3.19)]{Xiang18}),  we have
\begin{equation}\label{reg-v7}
    \begin{split}
    kd\int_{\partial\Omega}|\nabla v|^{2k-2}\frac{\partial|\nabla v|^2}{\partial\nu}&\leq 2kd\sigma_\Omega \||\nabla v|^k\|_{L^2(\partial\Omega)}^2\\
        &\leq k(k-1)d_\Omega\int_\Omega|\nabla v|^{2(k-2)}|\nabla|\nabla v|^2|^2+c_7d_\Omega\int_\Omega|\nabla v|^{2k},
    \end{split}
\end{equation}
where $d_\Omega$ is defined by \eqref{d-omega}.
By \eqref{inter-3}, Young's inequality with $\varepsilon$ and the fact $\|v(\cdot,t)\|_{L^\infty}\leq M_0$ in \eqref{b-v1}, we infer that
\begin{equation}\label{nonconvex}
\begin{split}
 c_7d_\Omega\int_\Omega|\nabla v|^{2k}&\leq \frac{kd_\Omega}{2\left(2k+\sqrt{n}\right)^2M_0^2}\int_\Omega|\nabla v|^{2(k+1)}
 +c_8d_\Omega M_0^{2k}\\
 &\leq \frac{k}{2}d_\Omega\int_\Omega|\nabla v|^{2(k-1)}|D^2v|^2+c_8d_\Omega M_0^{2k}.
 \end{split}
\end{equation}
Finally, substituting \eqref{int-coro-7+}, \eqref{reg-v7} and \eqref{nonconvex} into \eqref{Key1} and then keeping key parameters like $\xi_2$ and $M_0$, we accomplish our desired estimate  \eqref{ve-1}.
\end{proof}

\begin{lemma}\label{high-reg}
  Let $\Omega\subset\mathbb{R}^n(n\geq 1)$ and $(u,v,w)$ be the solution of \eqref{system} obtained in Lemma \ref{LS}.  Then, for any $k\geq 1$, there exist two positive constants
$\xi_*(k)$ and $\mu_*(k)$ defined respectively by \eqref{xi1-large0} and \eqref{mu1-large0} such that whenever one of the following conditions holds:
\begin{equation}\label{bdd-con0}
 (1) \ \xi_1\geq \xi_*(k)\chi^2;\ \  \  (2)\  \theta=1, \  \mu\geq \max\left\{1, \ \chi^\frac{8+2n}{5+n}\right\}\mu_*(k)\chi^\frac{2}{5+n}; \ \ \  (3)\  \theta>1, \ \mu>0,
\end{equation}
there exist three positive constants $C_2,C_3, C_4$ independent of $t,\chi, \mu, d$ and $\xi_i~(i=1,2)$ but  depending on $u_0, v_0, \theta, a, k, n$ and $\Omega$ such that, for $t\in (0, T_{\text{max}})$,
\begin{equation}\label{lkuv}
\int_\Omega u^k\leq C_2\left(1+m_1^k+M_1(k)+(1+d_\Omega M_0^{2k})d^{-1}\chi^2M_0^{2-k}+M_2^c(k)\right):=C_2M_2(k),
\end{equation}
and
\begin{equation}\label{l2kv}
\int_\Omega |\nabla v|^{2k}\leq \frac{C_3M_2(k)}{\color{black}{(1+M_0\xi_2)^{-k}d^{-1}\chi^2M_0^{2-k}}},
\end{equation}
as well as, if $k>n$,
\begin{equation}\label{gradw-linfty}
\|w(\cdot,t)\|_{W^{1,\infty}}\leq C_4M_2^\frac{1}{k}(k),
\end{equation}
where $M_2^c$  is defined by
\begin{equation}\label{M2*}
     M_2^c(k):=\begin{cases}
    0,\quad \quad \quad \quad \quad \quad  \quad  \text{if  } \xi_1\geq\xi_*(k)\chi^2,\\[0.25cm]
    0,\quad \quad \quad   \quad \quad \quad  \theta=1, \text{ if }   \mu\geq \max\left\{1, \ \chi^\frac{8+2n}{5+n}\right\}\mu_*(k)\chi^\frac{2}{5+n},\\[0.25cm] \frac{(\theta-1)}{\mu^\frac{k+1}{\theta-1}}
\left[(1+\frac{1}{d^{k+1}})(1+M_0\xi_2)M_0\chi^2 \right]^\frac{k+\theta}{\theta-1},\ \text{if } \theta>1,~\mu>0.
    \end{cases}
\end{equation}
\end{lemma}
\begin{proof}
We apply the estimate  \eqref{inter-3} and  the fact $\|v(\cdot,t)\|_{L^\infty}\leq M_0$ to deduce that
\begin{equation}\label{high-reg1}
\begin{split}
\frac{\chi^2(k-1)}{2}\int_\Omega u^k|\nabla v|^2
&\leq \varepsilon_1 \int_\Omega u^{k+1}+c_9\chi^{2(k+1)}\varepsilon_1^{-k}\int_\Omega |\nabla v|^{2(k+1)}\\
&\leq \varepsilon_1 \int_\Omega u^{k+1}+c_{10}\chi^{2(k+1)}M_0^2\varepsilon_1^{-k}\int_\Omega|\nabla v|^{2(k-1)}|D^2v|^2 ,\\
\end{split}
\end{equation}
where  $\varepsilon_1>0$, to  be chosen  in \eqref{epsi-def} below.
On the one hand, using the widely known Gagliardo-Nirenberg inequality (cf. \cite[Section 3]{TW-M3AS-2013}), we can find a constant $c_{11}>0$ only depending on $k, n$ and $\Omega$  and $\alpha=\frac{\frac{k}{2}-\frac{1}{2}}{\frac{k}{2}-\frac{1}{2}+\frac{1}{n}}\in(0,1)$  such that
\begin{equation}\label{GN-u}
\begin{split}
\int_\Omega u^k=\|u^{\frac{k}{2}}\|^2_{L^2}&\leq c_{11}(\|\nabla u^{\frac{k}{2}}\|^{2\alpha}_{L^2}\|u^{\frac{k}{2}}\|^{2(1-\alpha)}_{L^{\frac{2}{k}}}
+\|u^{\frac{k}{2}}\|^2_{L^{\frac{2}{k}}})  \\
&\leq c_{11}\left(\frac{k^2}{4}\int_\Omega u^{k-2}|\nabla u|^2\right)^\alpha m_1^{k(1-\alpha)}+c_{11}m_1^k  \\
&\leq \frac{k(k-1)}{2}\int_\Omega u^{k-2}|\nabla u|^2+c_{12}m_1^k.
\end{split}
\end{equation}
Then  substituting \eqref{high-reg1} and \eqref{GN-u} into \eqref{cee-1} with $p=k$, one has
\begin{equation}\label{high-reg3}
    \begin{split}
         &\frac{d}{dt}\int_\Omega u^k+{\color{black}\int_\Omega u^k}+\frac{\xi_1(k-1)}{\color{black}2}
    \int_\Omega u^{k+1}+\frac{k\mu}{2} \int_\Omega u^{\theta+k}\\
        &\leq \varepsilon_1\int_\Omega u^{k+1}+ c_{10}\chi^{2(k+1)}M_0^2\varepsilon_1^{-k}\int_\Omega|\nabla v|^{2(k-1)}|D^2v|^2+c_{12}m_1^k+M_1(k).
    \end{split}
\end{equation}
On the other hand, it follows from  Lemma \ref{ve} and the fact $d_\Omega\leq d$  that
\begin{equation}\label{high-reg4}
    \begin{split}
        &\frac{d}{dt}\int_\Omega|\nabla v|^{2k}+2k\int_\Omega|\nabla v|^{2k}+
         kd\int_\Omega|\nabla v|^{2k-2}|D^2v|^2\\
        &\leq  \frac{C_1}{d^k}(1+M_0\xi_2)^{k+1}M_0^{k-1}\int_\Omega u^{k+1} +C_1d_\Omega M_0^{2k}.
    \end{split}
\end{equation}
Let us first set
\begin{equation}\label{epsi-def}
c_{13}=\frac{C_1c_{10}}{k}, \ \  \ \ \varepsilon_1=(1+M_0\xi_2)M_0\chi^2(kc_{13})^\frac{1}{k+1},
\end{equation}
and then we put
\begin{equation}\label{delta-epsi}
\delta(\varepsilon_1)=\frac{c_{10}\chi^{2(k+1)}M_0^2\varepsilon_1^{-k}}{kd}= \frac{c_{10}}{k}(1+M_0\xi_2)^{-k}d^{-1}\chi^2M_0^{2-k}(kc_{13})^{-\frac{k}{k+1}},
\end{equation}
and
\begin{equation}\label{yt}
y(t):=\int_\Omega u^k+\delta(\varepsilon_1)\int_\Omega|\nabla v|^{2k}.
\end{equation}
Then we multiply \eqref{high-reg4} by $\delta(\varepsilon_1)$ and add the result to \eqref{high-reg3} to obtain
\begin{equation}\label{high-reg6}
\begin{split}
    &\ \ \ \ y'(t)+y(t)+\frac{\xi_1(k-1)}{2}
    \int_\Omega u^{k+1}+\frac{k\mu}{2} \int_\Omega u^{\theta+k}\\
    &\leq \left\{c_{13}\left(\frac{1}{d}(1+M_0\xi_2)M_0\chi^2\right)^{k+1}\varepsilon_1^{-k}+\varepsilon_1\right\}\int_\Omega u^{k+1}+c_{14}\\
    &=(kc_{13})^\frac{1}{k+1}\left(1+\frac{1}{kd^{k+1}}\right)(1+M_0\xi_2)M_0\chi^2\int_\Omega u^{k+1}+
     c_{14} ,\\
    \end{split}
\end{equation}
where
\begin{equation}\label{c5-def-bdd}
c_{14}=c_{12}m_1^k+M_1(k)+\delta(\varepsilon_1)C_1d_\Omega M_0^{2k}.
\end{equation}
To  show that the first term on the right-hand side of \eqref{high-reg6} can be absorbed  by the terms involving $\int_\Omega u^{k+1}$ or $\int_\Omega u^{\theta+k}$, we first define
\begin{equation}\label{xi1-large0}
\xi_*(k)=\frac{2(kc_{13})^{\frac{1}{k+1}}}{k-1}\left(1+\frac{1}{kd^{k+1}}\right)(1+M_0\xi_2)M_0.
\end{equation}
{\bf Case 1: $\xi_1\geq \xi_*(k)\chi^2$}. In this case, notice that $M_0$ is independent of $\xi_1$ by \eqref{b-v1} and \eqref{b-uvw1}, in view of \eqref{high-reg6} along with \eqref{c5-def-bdd} and  \eqref{xi1-large0}, we deduce that
\begin{equation}\label{bdd-odi}
    y'(t)+y(t)\leq c_{12}m_1^k+M_1(k)+\delta(\varepsilon_1)C_1d_\Omega M_0^{2k}.
\end{equation}
Solving this ODI and using \eqref{delta-epsi} and \eqref{yt},  we get  that
\begin{equation}\label{high-reg7}
    \begin{split}
    &\int_\Omega u^k+\frac{c_{10}}{k}(1+M_0\xi_2)^{-k}d^{-1}\chi^2M_0^{2-k}(kc_{13})^{-\frac{k}{k+1}}\int_\Omega |\nabla v|^{2k}\\[0.25cm]
    &\leq \|u_0\|^k_{L^k}+c_{15}(1+M_0\xi_2)^{-k}d^{-1}\chi^2M_0^{2-k}\|\nabla v_0\|^{2k}_{L^{2k}}\\[0.25cm]
    &\ \ +c_{12}m_1^k+M_1(k)+c_{15}(1+M_0\xi_2)^{-k}d^{-1}\chi^2M_0^{k+2}d_\Omega \\[0.25cm]
    &\leq c_{16}\left[1+m_1^k+M_1(k)+(1+d_\Omega M_0^{2k})d^{-1}\chi^2M_0^{2-k}\right].
    \end{split}
\end{equation}
{\bf Case 2: $\theta=1$ and $\mu>0$ is suitably large.}  In this case, we first observe from  \eqref{b-v1} and \eqref{b-uvw1} that $M_0$ is bounded by  $O(1)(1+\mu^{-(2+n/2)})$, and therefore,
$$
(1+M_0\xi_2)M_0\leq O(1)\left(1+\left(\frac{1}{\mu}\right)^{4+n}\right).
$$
This enables us to infer that
\begin{equation}\label{m1k-def}
\hat{\mu}_*(k)=\frac{2(kc_{13})^{\frac{1}{k+1}}}{k}\left(1+\frac{1}{kd^{k+1}}\right) \sup_{0<\mu<1}\left\{\mu^{4+n}(1+M_0\xi_2)M_0\right\}<+\infty
\end{equation}
and
\begin{equation}\label{m2k-def}
\tilde{\mu}_*(k)=\frac{2(kc_{13})^{\frac{1}{k+1}}}{k}\left(1+\frac{1}{kd^{k+1}}\right)
\sup_{\mu\geq 1}\left\{(1+M_0\xi_2)M_0\right\}<+\infty.
\end{equation}
Now, we define
\begin{equation}\label{mu1-large0}
\mu_*(k)=\max\left\{\left(\hat{\mu}_*(k)\right)^\frac{1}{5+n}, \ \ \ \tilde{\mu}_*(k)\right\}<+\infty.
\end{equation}
Then,  under (2) of \eqref{bdd-con0}, we see from  \eqref{mu1-large0} that
$$
 \mu\geq \max\left\{1, \ \chi^\frac{8+2n}{5+n}\right\}\mu_*(k)\chi^\frac{2}{5+n}
 \geq \max\left\{\left(\hat{\mu}_*(k)\right)^\frac{1}{5+n}\chi^\frac{2}{5+n}, \ \  \tilde{\mu}_*(k) \chi^2\right\},
$$
and so \eqref{m1k-def} together with \eqref{m2k-def} implies
\begin{equation}\label{odi-cond}
\mu\geq \frac{2(kc_{13})^{\frac{1}{k+1}}}{k}\left(1+\frac{1}{kd^{k+1}}\right)
(1+M_0\xi_2)M_0\chi^2.
\end{equation}
Then in light of  \eqref{high-reg6} along with \eqref{c5-def-bdd} and  \eqref{odi-cond}, we derive  an identical  ODI as \eqref{bdd-odi}, and then we get the same estimate as \eqref{high-reg7}.

{\bf Case 3: $\theta>1$ and $\mu>0$.}
Let
\begin{equation}\label{c7-def-bdd}
\begin{split}
c_{17}&:=\sup_{s>0}\left\{(kc_{13})^{\frac{1}{k+1}}\left(1+\frac{1}{kd^{k+1}}\right)(1+M_0\xi_2)M_0 \chi^2 s^{k+1}-\frac{k\mu}{2}s^{k+\theta} \right\}\\
&=\frac{(\theta-1)}{k+\theta}\left(\frac{2(k+1)}{k(k+\theta)\mu}\right)^\frac{k+1}{\theta-1}
\left[(kc_{13})^{\frac{1}{k+1}}\left(1+\frac{1}{kd^{k+1}}\right)(1+M_0\xi_2)M_0\chi^2 \right]^\frac{k+\theta}{\theta-1}.
\end{split}
\end{equation}
 Then, combining \eqref{c7-def-bdd} and  \eqref{high-reg6} along with \eqref{c5-def-bdd}, we end up with
$$
    y'(t)+y(t)\leq c_{12}m_1^k+M_1(k)+\delta(\varepsilon_1)C_1d_\Omega M_0^{2k}+c_{17} |\Omega|,
$$
which in conjunction with  \eqref{high-reg7},  \eqref{c7-def-bdd} and \eqref{delta-epsi}  allows us to deduce that
\begin{equation}\label{high-reg11}
\begin{split}
&\int_\Omega u^k+\frac{c_{10}}{k}(1+M_0\xi_2)^{-k}d^{-1}\chi^2M_0^{2-k}(kc_{13})^{-\frac{k}{k+1}}\int_\Omega |\nabla v|^{2k}\\[0.25cm]
&\leq c_{16}\left[1+m_1^k+M_1(k)+(1+d_\Omega M_0^{2k})d^{-1}\chi^2M_0^{2-k}\right]\\[0.25cm]
&\ \ \ +c_{18}\frac{(\theta-1)}{\mu^\frac{k+1}{\theta-1}}
\left[\left(1+\frac{1}{d^{k+1}}\right)(1+M_0\xi_2)M_0\chi^2 \right]^\frac{k+\theta}{\theta-1}.
\end{split}
\end{equation}
The desired qualitative $(L^k, L^{2k})$-bounds of $(u, \nabla v)$ in \eqref{lkuv} and \eqref{l2kv} are resulted from  \eqref{high-reg7} and \eqref{high-reg11} upon simple algebraic manipulations.  Finally, if $k>n$, then the $W^{2,k}$-estimate \eqref{int-coro-3} and the Sobolev embedding $W^{2,k}\hookrightarrow W^{1,\infty}$ together imply \eqref{gradw-linfty}.
\end{proof}

\subsection{Qualitative $(L^\infty, W^{1,\infty})$-boundedness of $(u,v)$}
Before proceeding, based on  \eqref{xi1-large0} and \eqref{mu1-large0},  we first  define
\begin{equation}\label{xi1-large}
\xi_0=\inf_{k>\frac{n}{2}}\xi_*(k)
=\inf_{k>\frac{n}{2}}\left\{\frac{2(kc_{13})^\frac{1}{k+1}}{k-1}
\left(1+\frac{1}{kd^{k+1}}\right)\right\}(1+M_0\xi_2)M_0
\end{equation}
and, with $\hat{\mu}_*(k)$ and $\tilde{\mu}_*(k)$ defined by \eqref{m1k-def} and \eqref{m2k-def},
\begin{equation}\label{mu1-large}
\mu_0=\inf_{k>\frac{n}{2}}\mu_*(k)=\inf_{k>\frac{n}{2}}\left\{\max\left\{\left(\hat{\mu}_*(k)\right)^\frac{1}{5+n}, \ \ \ \tilde{\mu}_*(k)\right\}\right\}.
\end{equation}
\begin{lemma}\label{b-u}
Let $\Omega\subset\mathbb{R}^n(n\geq 2)$ and $(u,v,w)$ be the local solution of \eqref{system} obtained in Lemma \ref{LS}.  Suppose that  one of the following conditions holds:
\begin{equation}\label{bdd-con}
 (1) \ \xi_1>\xi_0\chi^2;\   \   (2)\  \theta=1, \  \mu> \max\left\{1, \ \chi^\frac{8+2n}{5+n}\right\}\mu_0\chi^\frac{2}{5+n}; \   \  (3)\  \theta>1, \ \mu>0.
\end{equation}
 Then the $L^\infty$-norm of the $u$-solution component is uniformly bounded on $(0, T_{\text{max}})$. In particular, if one of the following conditions holds:
\begin{equation}\label{bdd-con+}
 (1) \ \xi_1\geq \xi_*(3n)\chi^2;\ \    (2) \ \theta=1, \  \mu\geq \max\left\{1, \ \chi^\frac{8+2n}{5+n}\right\}\mu_*(3n) \chi^\frac{2}{5+n}; \ \   (3)\  \theta>1,   \mu>0,
\end{equation}
 then there exists $C_5>0$  independent of $t,\chi,\mu,d~and~\xi_i(i=1,2)$ but depending on  $n,u_0, v_0$ and $\Omega$ such that the local solution  $(u,v,w)$ of \eqref{system} obtained in Lemma \ref{LS} satisfies, for $t\in (0, T_{\text{max}})$, that
\begin{equation}\label{u-infty-fin}
\begin{split}
&\|u(t)\|_{L^\infty}\leq C_5M_3(3n)\\
&:= C_5 \begin{cases}
1+M_2^\frac{1}{3n}(3n)+\xi_1 M_2^\frac{2}{3n}(3n)+\frac{\chi M_2^\frac{2}{3n}(3n)}{\left((1+M_0\xi_2)^{-3n}
d^{-1}\chi^2M_0^{2-3n}\right)^\frac{1}{3n}}, &\text{if } a=\mu=0, \\[0.25cm]
1+\left(\frac{1}{\mu}\right)^\frac{1}{\theta}+\xi_1 M_2^\frac{2}{3n}(3n) +\frac{\chi M_2^\frac{2}{3n}(3n)}{\left((1+M_0\xi_2)^{-3n}
d^{-1}\chi^2M_0^{2-3n}\right)^\frac{1}{3n}},  &\text{if } a,  \mu>0.
\end{cases}
\end{split}
\end{equation}
where $M_2(\cdot)$  is defined by \eqref{M2*}.
\end{lemma}

\begin{proof}
By the definitions of $\xi_0$ and $\mu_0$ respectively in \eqref{xi1-large} and \eqref{mu1-large}, it is  easy to see that the corresponding case of \eqref{bdd-con} implies that of \eqref{bdd-con0} for some $k>\frac{n}{2}$. Then an application of Lemma \ref{high-reg} shows
\begin{equation}\label{ulp}
\|u(\cdot,t)\|_{L^k}\leq c_{19}M_2^\frac{1}{k}(k), \quad \quad  \|\nabla v(\cdot,t)\|_{L^{2k}}\leq c_{20}\left(\frac{M_2(k)}{(1+M_0\xi_2)^{-k}
d^{-1}\chi^2M_0^{2-k}}\right)^\frac{1}{2k}.
\end{equation}
Therefore, the $W^{2,k}$-estimate \eqref{int-coro-3} yields
$$
\|w(\cdot,t)\|_{W^{2,k}}\leq c_{20}\|u(\cdot,t)-\bar{u}\|_{L^k}\leq 2c_{20}\|u(\cdot,t)\|_{L^k}\leq 2c_{19}c_{20}M_2^\frac{1}{k}(k),
$$
and so the Sobolev embedding $W^{2,k}\hookrightarrow W^{1,q}$ for some  $q>n$ shows
\begin{equation}\label{wlp}
\|w(\cdot,t)\|_{W^{1,q}}\leq  c_{21}M_2^\frac{1}{k}(k), \quad \quad     q=\frac{4nk}{k+3(n-k)^+}.
\end{equation}
We now use similar sprits as used in \cite{JX18, Xiang18} to derive the $L^\infty$-bound of $u$. To that purpose, we  use the  variation-of-constants formula to the $u$-equation in \eqref{system} to write
\begin{equation*}
\begin{split}
u(\cdot,t)&=e^{t(\Delta-1)}u_0+\int_0^te^{(\Delta-1)(t-s)}\nabla\cdot\{u(\cdot,s)(\xi_1\nabla w(\cdot,s)-\chi\nabla v(\cdot,s))\}ds\\
&\;\;\; + \int_0^te^{(\Delta-1)(t-s)}\left[(a+1) u(\cdot,s) -\mu u^{\theta+1}(\cdot,s)\right]ds\\
&=:u_1(\cdot,t)+u_2(\cdot,t)+u_3(\cdot,t).
\end{split}
\end{equation*}
We next estimate $u_1, u_2$ and $u_3$. First, the nonnegativity of $u$ shows, for $t\in(0, T_{\text{max}})$,
\begin{equation}\label{u-infty-dec}
\|u(\cdot,t)\|_{L^\infty}=\sup_{x\in\Omega }u(x,t)\leq \sup_{x\in\Omega }u_1(x,t)+\sup_{x\in\Omega }u_2(x,t)+\sup_{x\in\Omega }u_3(x,t).
\end{equation}
By the order property of the Neumann heat semigroup  $(e^{t\Delta})_{t>0}$ due to the maximum
principle, we estimate $u_1$ and $u_3$ in the following ways:
\begin{equation}\label{u1-est}
\|u_1(\cdot, t)\|_{L^\infty} = \| e^{t(\Delta-1)}u_0\|_{L^\infty} \leq  e^{-t}\| u_0\|_{L^\infty}  \leq  \|u_0\|_{L^\infty},
\end{equation}
and, the semigroup estimate \eqref{SG-1} in Lemma \ref{SG} and \eqref{ulp} yield
\begin{equation}\label{u3-est}
\begin{split}
u_3(\cdot, t)  &= \int_0^te^{(\Delta-1)(t-s)}\left[(a+1) u(\cdot,s) -\mu u^{\theta+1}(\cdot,s)\right]ds\\[0.25cm]
&\leq \begin{cases}\int_0^te^{(\Delta-1)(t-s)} u(\cdot,s)ds, & \text{ if } a=\mu=0, \\[0.25cm]
\int_0^te^{(\Delta-1)(t-s)}\left(\frac{a+1}{(1+\theta)\mu}\right)^\frac{1}{\theta}
\frac{(a+1)\theta}{1+\theta} ds,  & \text{ if } a\geq 0,\ \  \mu>0,
\end{cases}\\[0.25cm]
&\leq c_{21} \begin{cases}M_2^\frac{1}{k}(k)\int_0^t(1+(t-s)^{-\frac{n}{2k}})e^{-(t-s)}ds, & \text{ if } a=\mu=0, \\[0.25cm]
\int_0^te^{-(t-s)}e^{\Delta (t-s)}\left(\frac{1}{\mu}\right)^\frac{1}{\theta}ds,  & \text{ if } a\geq 0,\ \ \mu>0,
\end{cases}\\[0.25cm]
&\leq c_{22} \begin{cases}M_2^\frac{1}{k}(k), & \text{ if } a=\mu=0, \\[0.25cm]
\left(\frac{1}{\mu}\right)^\frac{1}{\theta},  & \text{ if } a\geq 0,\ \   \mu>0.
\end{cases}
\end{split}
\end{equation}
For convenience of reference, we  list two H\"{o}lder's type interpolation inequalities:
$$
\|fg\|_{L^r}\leq \|f\|_{L^p}\|g\|_{L^q} \text{ with } p,q,r\geq 1,  \ \   \frac{1}{r}=\frac{1}{p}+\frac{1}{q},
$$
and
$$
\|f\|_{L^r}\leq \|f\|_{L^\infty}^\frac{(r-k)^+}{r}\|f\|_{L^k}^\frac{k}{k+(r-k)^+}
|\Omega|^\frac{(k-r)^+}{kr}  \text{ with } k,r\geq 1.
$$
Employing the semigroup estimate \eqref{SG-4} in Lemma \ref{SG},  H\"{o}lder's interpolation inequalities above and  \eqref{wlp}, we deduce, for $k>\frac{n}{2}$, that
\begin{equation}\label{u21-est}
\begin{split}
&\left\|\xi_1\int_0^te^{(\Delta-1)(t-s)}\nabla \cdot(u\nabla w)ds\right\|_{L^\infty}\\
&\leq c_{23} \xi_1\int_0^t\left(1+(t-s)^{-1+\frac{k-(n-k)^+}{8k}}\right)
e^{-(t-s)}\|u\nabla w\|_{L^\frac{4nk}{3k+(n-k)^+}}ds\\
&\leq c_{24} \xi_1\int_0^t\left(1+(t-s)^{-1+\frac{k-(n-k)^+}{8k}}\right)
e^{-(t-s)}\|u\|_{L^\frac{2nk}{k-(n-k)^+}} \|\nabla w\|_{L^\frac{4nk}{k+3(n-k)^+}}ds\\
&\leq c_{25} \xi_1M_2^\frac{2k-2(n-k)^++\left(2n-k+(n-k)^+\right)^+}{\left[k-(n-k)^++\left(2n-k+(n-k)^+\right)^+\right]k}(k)\left(\sup_{s\in(0,t)}\|u(s)\|_{L^\infty}\right)^
\frac{\left(2n-k+(n-k)^+\right)^+}{2n}\\
&= c_{25} \xi_1\begin{cases}
M_2^\frac{2n+k-(n-k)^+}{2nk}(k)\left(\sup_{s\in(0,t)}\|u(s)\|_{L^\infty}\right)^
\frac{2n-k+(n-k)^+}{2n}, &\text{ if } k<2n, \\[0.25cm]
M_2^\frac{2}{k}(k),  &\text{ if } k\geq 2n,
\end{cases} \\
&\leq \begin{cases}
\frac{1}{4}\sup_{s\in(0,t)}\|u(s)\|_{L^\infty}+c_{26}\xi_1^\frac{2n}{k-(n-k)^+}
M_2^\frac{2n+k-(n-k)^+}{[k-(n-k)^+]k}(k), &\text{ if } k<2n, \\[0.25cm]
c_{25} \xi_1M_2^\frac{2}{k}(k),  &\text{ if } k\geq 2n,
\end{cases}
\end{split}
\end{equation}
where we used the fact that
\begin{align*}
&\|u\|_{L^\frac{2nk}{k-(n-k)^+}}\|\nabla w\|_{L^\frac{4nk}{k+3(n-k)^+}}\\
&\leq c_{27}\|u\|_{L^k}^\frac{k-(n-k)^+}{k-(n-k)^++\left(2n-k+(n-k)^+\right)^+}
\|u\|_{L^\infty}^\frac{\left(2n-k+(n-k)^+\right)^+}{2n}\|\nabla w\|_{L^\frac{4nk}{k+3(n-k)^+}}\\
&\leq c_{28}M_2^\frac{2k-2(n-k)^++\left(2n-k+(n-k)^+\right)^+}{\left[k-(n-k)^++\left(2n-k+(n-k)^+\right)^+\right]k}
\|u\|_{L^\infty}^\frac{\left(2n-k+(n-k)^+\right)^+}{2n}
\end{align*}
and the finiteness of gamma integral due to the fact that $k-(n-k)^+>0$:
$$
\int_0^t\left(1+(t-s)^{-1+\frac{k-(n-k)^+}{8k}}\right)
e^{-(t-s)}ds=\int_0^t\left(1+\tau^{-1+\frac{k-(n-k)^+}{8k}}\right)
e^{-\tau}d\tau<\infty.
$$
Similarly, using the boundedness information in \eqref{ulp}, we infer
\begin{equation}\label{u22-est}
\begin{split}
&\left\|-\chi\int_0^te^{(\Delta-1)(t-s)}\nabla \cdot(u\nabla v)(\cdot,s)ds\right\|_{L^\infty}\\
&\leq c_{29} \chi\int_0^t\left(1+(t-s)^{-1+\frac{2k-n}{8k}}\right)
e^{-(t-s)}\|(u\nabla v)(\cdot,s)\|_{L^\frac{4nk}{2k+n}}ds\\
&\leq c_{30} \chi\int_0^t\left(1+(t-s)^{-1+\frac{2k-n}{8k}}\right)
e^{-(t-s)}\|u(\cdot,s)\|_{L^\frac{4nk}{2k-n}} \|\nabla v(\cdot,s)\|_{L^{2k}}ds\\
&\leq c_{31} \frac{\chi M_2^{\frac{4k-2n+(5n-2k)^+}{k[2k-n+(5n-2k)^+]}}(k)}{\left((1+M_0\xi_2)^{-k}
d^{-1}\chi^2M_0^{2-k}\right)^\frac{1}{k}}\left(\sup_{s\in(0,t)}\|u(\cdot,s)\|_{L^\infty}\right)^
\frac{(5n-2k)^+}{4n}\\
&\leq \begin{cases}
\frac{1}{4}\sup_{s\in(0,t)}\|u(s)\|_{L^\infty}+c_{32}\frac{\chi^\frac{4n}{2k-n} M_2^{\frac{2k+3n}{k(2k-n)}}(k)}{\left((1+M_0\xi_2)^{-k}
d^{-1}\chi^2M_0^{2-k}\right)^\frac{4n}{k(2k-n)}}, &\text{ if } k<\frac{5n}{2}, \\[0.25cm]
c_{31}\frac{ \chi M_2^\frac{2}{k}(k)}{\left((1+M_0\xi_2)^{-k}
d^{-1}\chi^2M_0^{2-k}\right)^\frac{1}{k}},  &\text{ if } k\geq \frac{5n}{2}.
\end{cases}
\end{split}
\end{equation}
In the case of $k\geq \frac{5n}{2}$, substituting \eqref{u1-est}, \eqref{u3-est}, \eqref{u21-est} and \eqref{u22-est} into \eqref{u-infty-dec}, we accomplish that
\begin{equation}\label{u-fin1}
\|u(t)\|_{L^\infty}\leq c_{32}\begin{cases}
1+M_2^\frac{1}{k}(k)+\xi_1 M_2^\frac{2}{k}(k)+\frac{\chi M_2^\frac{2}{k}(k)}{\left((1+M_0\xi_2)^{-k}
d^{-1}\chi^2M_0^{2-k}\right)^\frac{1}{k}}, &\text{if } \mu=0, \\[0.25cm]
1+\left(\frac{1}{\mu}\right)^\frac{1}{\theta}+\xi_1M_2^\frac{2}{k}(k)+\frac{\chi M_2^\frac{2}{k}(k)}{\left((1+M_0\xi_2)^{-k}
d^{-1}\chi^2M_0^{2-k}\right)^\frac{1}{k}},  &\text{if }   \mu>0.
\end{cases}
\end{equation}
Similarly, in the case of $k\in[2n, \frac{5n}{2})$, we have
\begin{equation}\label{u-fin2}
\begin{split}
\|u(t)\|_{L^\infty}\leq c_{33}\begin{cases}
1+M_2^\frac{1}{k}(k)+\xi_1M_2^\frac{2}{k}(k)+\frac{\chi^\frac{4n}{2k-n} M_2^{\frac{2k+3n}{k(2k-n)}}(k)}{\left((1+M_0\xi_2)^{-k}
d^{-1}\chi^2M_0^{2-k}\right)^\frac{4n}{k(2k-n)}}, &\text{if } \mu=0, \\[0.2cm]
1+\left(\frac{1}{\mu}\right)^\frac{1}{\theta}+\xi_1M_2^\frac{2}{k}(k)+\frac{\chi^\frac{4n}{2k-n} M_2^{\frac{2k+3n}{k(2k-n)}}(k)}{\left((1+M_0\xi_2)^{-k}
d^{-1}\chi^2M_0^{2-k}\right)^\frac{4n}{k(2k-n)}},  &\text{if }  \mu>0.
\end{cases}
\end{split}
\end{equation}
And in the case of $k\in(\frac{n}{2}, 2n)$, we get
\begin{equation}\label{u-fin3}
 \begin{split}
\|u(t)\|_{L^\infty}\leq & c_{34}\left(1+\frac{\chi^\frac{4n}{2k-n} M_2^{\frac{2k+3n}{k(2k-n)}}(k)}{\left((1+M_0\xi_2)^{-k}
d^{-1}\chi^2M_0^{2-k}\right)^\frac{4n}{k(2k-n)}}\right)\\
&+c_{34}\xi_1^\frac{2n}{k-(n-k)^+}
M_2^\frac{2n+k-(n-k)^+}{[k-(n-k)^+]k}(k)+c_{34}\begin{cases}
M_2^\frac{1}{k}(k), &\text{ if } \mu=0, \\[0.25cm]
\left(\frac{1}{\mu}\right)^\frac{1}{\theta},  &\text{ if },  \mu>0.
\end{cases}
\end{split}
\end{equation}
Combining \eqref{u-fin1}, \eqref{u-fin2} and \eqref{u-fin3}, we see  either one of \eqref{bdd-con} implies that  $\|u(\cdot, t)\|_{L^\infty}$ is uniformly bounded for $t\in(0, T_{\text{max}})$. In particular, if one of  \eqref{bdd-con+} holds, then the desired qualitative  bound of $u$ in \eqref{u-infty-fin} follows upon setting $k=3n$ in \eqref{u-fin1}.
\end{proof}

\begin{lemma}\label{ev} Let $\Omega\subset\mathbb{R}^n(n\geq 2)$ and $(u,v,w)$ be the solution of \eqref{system} obtained in Lemma \ref{LS}.  If  one of \eqref{bdd-con}  holds, then $\|\nabla v(\cdot,t)\|_{L^\infty}$  is uniformly bounded on $(0, T_{\text{max}})$. In particular,  if one of the following conditions holds:
 \begin{equation}\label{bdd-con++}
 (1) \ \xi_1\geq \xi_*(3n)\chi^2;\ \    (2) \ \theta=1, \  \mu\geq \max\left\{1, \ \chi^\frac{8+2n}{5+n}\right\}\mu_*(3n) \chi^\frac{2}{5+n}; \ \   (3)\  \theta>1,   \mu>0,
\end{equation}
then, for $t\in (0, T_{\text{max}})$, there exists $C_6>0$ depending only on $n,u_0, v_0$ and $\Omega$ such that
\begin{equation}\label{gradv-est-fin}
\begin{split}
&\|\nabla v(\cdot, t)\|_{L^\infty}\\
&\leq C_6\left(1+\frac{\xi_2 M_2^\frac{1}{2n}(3n)}{d\left((1+M_0\xi_2)^{-3n}
d^{-1}\chi^2M_0^{2-3n}\right)^\frac{1}{6n}}+\frac{\xi_2}{d} M_0 M_2^\frac{1}{3n}(3n)+\frac{M_2^\frac{1}{3n}(3n)}{d}\right)
\end{split}
\end{equation}
and
\begin{equation}\label{w-1infty-est}
\|w(\cdot, t)\|_{W^{2,\infty}}\leq C_6M_3(3n),
\end{equation}
where $M_0$, $M_2(\cdot)$  and $M_3(\cdot)$ are defined by \eqref{b-v1}, \eqref{lkuv} and \eqref{u-infty-fin} respectively.
\end{lemma}
\begin{proof}
Using the time scaling $\tilde{t}=dt$, we rewrite the second equation in  \eqref{system} as follows:
$$
v_{\tilde{t}}=\Delta v-\frac{1}{d}v+\frac{\xi_2}{d}\nabla w\cdot\nabla v+\frac{\xi_2}{d} v(\bar{u}-u)+\frac{1}{d}u.
$$
Then an application of the variation-of-constants formula yields
\begin{equation}\label{ev-3}
\begin{split}
v(\cdot,\tilde{t})&=e^{(\Delta-\frac{1}{d})\tilde{t}}v_0+\frac{\xi_2}{d} \int_0^{\tilde{t}} e^{(\Delta-\frac{1}{d})(\tilde{t}-s)}(\nabla w\cdot\nabla v)(\cdot,s)ds\\
&\;\;\;\;+\frac{\xi_2}{d}\int_0^{\tilde{t}} e^{(\Delta-\frac{1}{d})(\tilde{t}-s)}v(\cdot,s)(\bar{u}(s)-u(\cdot,s))ds+\frac{1}{d}\int_0^{\tilde{t}} e^{(\Delta-\frac{1}{d})(\tilde{t}-s)}u(\cdot,s)ds.
\end{split}
\end{equation}
Under one of \eqref{bdd-con}, we first  see from Lemma \ref{b-u} and its proof, for some $k>\frac{n}{2}$, that
\begin{equation}\label{gradv-ets1}
\|u(\cdot,\tilde{t})\|_{L^\infty}+\|\nabla v(\cdot,\tilde{t})\|_{L^{2k}}\leq c_{35}, \quad \quad \forall \tilde{t}>0.
\end{equation}
Then the elliptic estimate applied to $\Delta w=\bar{u}-u$ along with Lemma \ref{b-v} shows that
\begin{equation}\label{gradv-ets2}
\|v(\cdot,\tilde{t})\|_{L^\infty}+\|\nabla w(\cdot,\tilde{t})\|_{L^\infty}\leq c_{36}, \quad \quad \forall \tilde{t}>0.
\end{equation}
Then, with \eqref{gradv-ets1} and \eqref{gradv-ets2} at hand, we can apply  the smoothing $L^p$-$L^q$ semigroup estimates in Lemma \ref{SG} to \eqref{ev-3} to infer that $\|\nabla v(\cdot,t)\|_{L^\infty}$ is uniformly bounded. We next aim to derive a qualitative bound for $\|\nabla v(\cdot,t)\|_{L^\infty}$ under one of \eqref{bdd-con++}.  In this circumstance,   Lemma \ref{high-reg} allows us to see, for $\tilde{t}\in (0, dT_{\text{max}})$, that
\begin{equation}\label{uln+gradvl2n}
\|u(\cdot,\tilde{t})\|_{L^{3n}}\leq c_{37}M_2^\frac{1}{3n}(3n), \   \ \    \|\nabla v (\cdot,\tilde{t})\|_{L^{6n}}\leq \frac{ c_{38} M_2^\frac{1}{6n}(3n)}{\left((1+M_0\xi_2)^{-3n}
d^{-1}\chi^2M_0^{2-3n}\right)^\frac{1}{6n}}.
\end{equation}
Hence, applying the $W^{2,\infty}$-estimate to \eqref{w-2p} and using the qualitative bound for $\|u\|_{L^\infty}$ in \eqref{u-infty-fin}, we find that
$$
 \|w(\cdot,\tilde{t})\|_{W^{2,\infty}}\leq c_{39}M_3(3n),  \ \  \tilde{t}\in (0, dT_{\text{max}}).
$$
This directly entails \eqref{w-1infty-est}.

Next, with the aid of \eqref{uln+gradvl2n} and  \eqref{gradw-linfty} with $k=3n$, and keeping  in mind that $\|v(\cdot,\tilde{t})\|_{L^\infty}\leq M_0$ by Lemma \ref{b-v}, we utilize  the H\"{o}lder  inequality and the semigroup estimates in Lemma \ref{SG}  to \eqref{ev-3} to deduce that
\begin{align*}
\|\nabla v(\cdot,\tilde{t})\|_{L^\infty}
&\leq \|\nabla e^{(\Delta-\frac{1}{d})\tilde{t}}v_0\|_{L^\infty}
+\frac{\xi_2}{d} \int_0^{\tilde{t}} \|\nabla e^{(\Delta-\frac{1}{d})(\tilde{t}-s)}(\nabla w\cdot\nabla v)(\cdot,s)\|_{L^\infty}ds \nonumber\\
&\ \ \ \ +\frac{\xi_2}{d}\int_0^{\tilde{t}}\|\nabla e^{(\Delta-\frac{1}{d})(\tilde{t}-s)}v(\cdot,s)(\bar{u}(s)-u(\cdot,s))\|_{L^\infty}
ds \nonumber\\
&\ \ \ \ +\frac{1}{d}\int_0^{\tilde{t}}\|\nabla e^{(\Delta-\frac{1}{d})(\tilde{t}-s)}u(\cdot,s)\|_{L^\infty}ds \nonumber\\
&\leq c_{40}\|\nabla v_0\|_{L^\infty}+\frac{c_{40}\xi_2}{d}\int_0^{\tilde{t}} (1+(\tilde{t}-s)^{-\frac{2}{3}}) e^{-\lambda_1(\tilde{t}-s)}\|\nabla v\|_{L^{6n}}\|\nabla w\|_{L^{6n}}ds\nonumber\\
&\ \ \ \ +\frac{c_{40}\xi_2}{d}\int_0^{\tilde{t}}(1+(\tilde{t}-s)^{-\frac{2}{3}})
e^{-\lambda_1(\tilde{t}-s)}\|v\|_{L^\infty}\|\bar{u}-u\|_{L^{3n}}ds \nonumber\\
&\ \ \ \ +\frac{c_{40}}{d}\int_0^{\tilde{t}}(1+(\tilde{t}-s)^{-\frac{2}{3}})
e^{-\lambda_1(\tilde{t}-s)}\|u\|_{L^{3n}}ds\\
&\leq c_{41}\Bigl(1+\frac{\xi_2 M_2^\frac{1}{2n}(3n)}{d\left((1+M_0\xi_2)^{-3n}
d^{-1}\chi^2M_0^{2-3n}\right)^\frac{1}{6n}}+\frac{\xi_2}{d} M_0 M_2^\frac{1}{3n}(3n)+\frac{M_2^\frac{1}{3n}(3n)}{d}\Bigr),
\end{align*}
which  is our desired qualitative bound  in  \eqref{gradv-est-fin} since $\tilde{t}\in (0, dT_{\text{max}})$ was arbitrary.
\end{proof}

\begin{proof}[{\bf Proof of Theorem \ref{GB}}]
 Based on the extensibility criterion \eqref{KS-1} in Lemma \ref{LS}, the qualitative boundedness described in Theorem \ref{GB} can be traced out from  Lemma \ref{b-v}, Lemma \ref{b-u} and Lemma \ref{ev}. The qualitative bound for $\|v\|_{L^\infty}$ in \eqref{v-infty-bdd} can be easily seen from Lemma  \ref{b-v}. The bounds for $\|\nabla w\|_{L^\infty}$ in \eqref{gradw-infty-bdd} comes mainly from Lemma \ref{high-reg} with $k=n+1$.
\end{proof}

\section{Higher order regularity of solutions}
In section 3, we have established the qualitative boundedness and thus global existence of solutions. To study large time behavior of global bounded solutions, we need further to  enhance regularity properties of bounded solutions.
\begin{lemma}\label{hg}
Let $(u,v,w)$ be the global and bounded classical solution of   \eqref{system} obtained in Theorem \ref{GB}. Then there exist $\sigma\in(0,1)$ and $C_7>0$ such that
\begin{equation}\label{hg-e1}
\|u(\cdot,t)\|_{C^{\sigma,\frac{\sigma}{2}}(\bar{\Omega}\times[t,t+1])}
\leq C_7, \quad \quad \forall t\geq1.
\end{equation}

\end{lemma}
\begin{proof}
In light of  Theorem \ref{GB}, the global  classical solution $(u,v, w)$ satisfies
\begin{equation}\label{hg-e2}
 u>0, \ v>0,  \   u+v+|\nabla v|+|\nabla w|\leq M \text{ on } \Omega\times (0, \infty).
\end{equation}
We now rewrite the first equation of  \eqref{system} in the following form:
\begin{equation*}
u_t=\nabla \cdot D(x,t,\nabla u)+R(x,t),  \quad \quad (x, t)\in\Omega\times (0, \infty),
\end{equation*}
where
$D(x,t,\eta)=\eta-\chi (u\nabla v)(x,t) +\xi_1 (u\nabla w)(x,t)$ and  $R(x,t)=(au-\mu u^{\theta+1})(x,t)$.

In view of the boundedness in \eqref{hg-e2} and the Young inequality, we readily deduce
 \begin{equation*}
 \begin{cases}
   D(x,t,\eta)\cdot\eta\geq \frac{1}{2}|\eta|^2-\frac{(\chi+\xi_1)^2M^4}{2}, \ \ |D(x,t,\eta)|\leq|\eta|+ (\chi+\xi_1)M^2, \\[0.25cm]
   |R(x,t)|\leq  aM+\mu M^{\theta+1}, \ \ \    \forall (x,t,\eta)\in \Omega\times (0, \infty)\times \mathbb{R}^n.
 \end{cases}
 \end{equation*}
Then \eqref{hg-e1} follows from  the H\"older regularity for  parabolic equations \cite[Theorem 1.3]{Po93}.
 \end{proof}

\begin{lemma}\label{up}
There exists a constant  $C_8>0$ such that the global bounded solution of \eqref{system} obtained in Theorem \ref{GB} fulfills
\begin{equation}\label{up-1}
    \|\nabla u(\cdot,t)\|_{L^{2n}}\leq C_8, \quad \quad \forall t\geq 1.
\end{equation}
\end{lemma}
\begin{proof}
Inspired from \cite[Section 3.3]{LX-20-CVPDE}, we first establish an $L^2$-bound for  $\Delta  v$. To this end, we  first recall from \eqref{hg-e2} and \eqref{cee-1} with $p=2$  of Lemma \ref{cee} that
\begin{equation}\label{ul2-est-f}
\frac{d}{dt}\int_\Omega u^2+\int_\Omega u^2+\int_\Omega |\nabla u|^2\leq c_{42}.
\end{equation}
Next, we take $\nabla $ to  the $v$-equation, and then take dot product with $-\nabla \Delta v$, and finally  use the $w$-equation and repeatedly use integration by parts  to derive from  \eqref{system} that
\begin{equation}\label{deltavl-2}
    \begin{split}
    &\frac{1}{2}\frac{d}{dt}\int_\Omega|\Delta v|^2+\int_\Omega|\Delta v|^2+d\int_\Omega|\nabla \Delta  v|^2\\
    &=-\xi_2 \int_\Omega\nabla\left(\nabla v\cdot\nabla w+\bar{u}v-uv\right)\cdot\nabla \Delta v-\int_\Omega\nabla u\cdot\nabla \Delta v\\
    &=-\xi_2 \int_\Omega\left(D^2 v\cdot \nabla w+D^2 w\cdot \nabla v \right)\cdot\nabla \Delta v+\xi_2\int_\Omega\left(u-\bar{u}\right)\nabla v\cdot\nabla \Delta v\\
    &\;\;\;\;\; +\int_\Omega\left(\xi_2 v-1\right)\nabla u\cdot\nabla \Delta v.
    \end{split}
\end{equation}
Based on \eqref{ul2-est-f} and \eqref{deltavl-2}, by   the boundedness \eqref{hg-e2},  the elliptic estimate $\|D^2 w\|_{L^2}\leq c_{43}\|u\|_{L^2}\leq c_{44}$,  the  Gagliardo-Nirenberg   interpolation inequality  and the $H^3$-elliptic estimate (cf. \cite[(4.19)]{Xiangjde}), we  obtain that
\begin{equation}\label{deltavl-21}
\begin{split}
&-\xi_2 \int_\Omega\left(D^2 v\cdot \nabla w+D^2 w\cdot \nabla v \right)\nabla \Delta v\\
&\leq M\xi_2\|D^2v\|_{L^2}\|\nabla \Delta v\|_{L^2}+\frac{d}{7}\|\nabla \Delta v\|_{L^2}^2+c_{45}\\
&\leq c_{46} \left(\|D^3v\|_{L^2}^\frac{2}{3}\|v\|_{L^2}^\frac{1}{3}+\|v\|_{L^2}\right)\|\nabla \Delta v\|_{L^2}+\frac{d}{7}\|\nabla \Delta v\|_{L^2}^2+c_{45}\\
&\leq c_{47}\|v\|_{H^3}^\frac{2}{3}\|\nabla \Delta v\|_{L^2}+\frac{d}{6}\|\nabla \Delta v\|_{L^2}^2+c_{47}\\
&\leq c_{48}\left(\|\Delta v\|_{H^1}+\|v\|_{L^2}\right)^\frac{2}{3}\|\nabla \Delta v\|_{L^2}+\frac{d}{6}\|\nabla \Delta v\|_{L^2}^2+c_{47}\\
&\leq c_{49} \left(\|\nabla \Delta v\|_{L^2}+\|v\|_{L^2}\right)^\frac{2}{3}\|\nabla \Delta v\|_{L^2}+\frac{d}{6}\|\nabla \Delta v\|_{L^2}^2+c_{47}\\
&\leq \frac{d}{5}\|\nabla \Delta v\|_{L^2}^2+c_{50}.
\end{split}
\end{equation}
In easier ways, we again use the boundedness \eqref{hg-e2} to estimate
\begin{equation}\label{deltavl-22}
\begin{split}
&\xi_2\int_\Omega\left(u-\bar{u}\right)\nabla v\cdot\nabla \Delta v+\int_\Omega\left(\xi_2 v-1\right)\nabla u\cdot\nabla \Delta v\\
&\leq \frac{3d}{5} \int_\Omega|\nabla \Delta v|^2+\frac{(1+M\xi_2)^2}{2d}\int_\Omega|\nabla u|^2+c_{51}.
    \end{split}
\end{equation}
Substituting \eqref{deltavl-21} and \eqref{deltavl-22}  into \eqref{deltavl-2}, we conclude that
\begin{equation}\label{deltavl-2f}
  \frac{d}{dt}\int_\Omega|\Delta v|^2+2\int_\Omega|\Delta v|^2\leq \frac{(1+M\xi_2)^2}{d}\int_\Omega|\nabla u|^2+2c_{50}+2c_{51}.
    \end{equation}
An obvious combination from \eqref{ul2-est-f}  and \eqref{deltavl-2f} enables one to derive that
$$
\frac{d}{dt}\int_\Omega\left( \frac{(1+M\xi_2)^2}{d}u^2+|\Delta v|^2\right)+\int_\Omega\left(\frac{(1+M\xi_2)^2}{d}u^2+|\Delta v|^2\right)\leq c_{52},
$$
which along with $H^2$-elliptic estimate  yields a uniform $H^2$-bound for $ v$:
\begin{equation} \label{delta-v-est}
\|D^2 v(\cdot,t)\|_{L^2}+\|\Delta v(\cdot,t)\|_{L^2}\leq c_{53}, \quad \quad \forall t\geq 1.
\end{equation}
Next, we proceed to derive an $L^2$-bound for $\nabla u$. For this, we compute from \eqref{system} that
\begin{equation}\label{ul2-id}
    \begin{split}
    &\frac{d}{dt}\int_\Omega|\nabla u|^2+\int_\Omega|\nabla u|^2+2\int_\Omega|\Delta  u|^2\\
    &=2\chi \int_\Omega\left(\nabla u\cdot\nabla v+u\Delta v\right)\Delta u-2\xi_1 \int_\Omega\left(\nabla u\cdot\nabla w+u\bar{u}-u^2\right)\Delta u\\
     &\ \ -2\int_\Omega (au-\mu u^{\theta+1})\Delta u+\int_\Omega|\nabla u|^2.
    \end{split}
\end{equation}
Using   the  boundedness \eqref{hg-e2},  \eqref{delta-v-est}, the Gagliardo-Nirenberg  inequality and the $H^2$-elliptic estimate, we deduce that
\begin{equation*}
    \begin{split}
    &2\chi \int_\Omega\left(\nabla u\nabla v+u\Delta v\right)\Delta u-2\xi_1 \int_\Omega\left(\nabla u\nabla w+u\bar{u}-u^2\right)\Delta u\\
     &\ \ -2\int_\Omega (au-\mu u^{\theta+1})\Delta u+\int_\Omega|\nabla u|^2\\
     &\leq \int_\Omega|\Delta u|^2+\left[1+4M^2(\chi^2+\xi^2_1)\right]\int_\Omega|\nabla u|^2+4M^2\chi^2\int_\Omega|\Delta v|^2+c_{54}\\
     &\leq  \int_\Omega|\Delta u|^2+c_{55}\left(\|D^2u\|_{L^2}^\frac{1}{2}\|u\|_{L^2}^\frac{1}{2}+\|u\|_{L^2}\right)^2+c_{55}\\
     &\leq  \int_\Omega|\Delta u|^2+c_{56}\|D^2u\|_{L^2}+c_{56}\\
     &\leq \int_\Omega|\Delta u|^2+c_{57}\left(\|\Delta u\|_{L^2}+\| u\|_{L^2}\right)+c_{57}\\
     &\leq 2\int_\Omega|\Delta u|^2+c_{58}.
    \end{split}
\end{equation*}
Substituting this into  \eqref{ul2-id} entails a simple ODI as follows:
   $$
    \frac{d}{dt}\int_\Omega|\nabla u|^2+\int_\Omega|\nabla u|^2\leq c_{58},
   $$
yielding immediately a uniform $L^2$-bound for $\nabla u$:
\begin{equation} \label{gradu-l2-est}
\|\nabla u(\cdot,t)\|_{L^2}\leq c_{59}, \quad \quad \forall t\geq 1.
\end{equation}

We are now ready to show the uniform boundedness of $\|\nabla u(\cdot,t)\|_{L^{2n}}$.  Again, we apply  integration by parts  to compute  from  \eqref{system} that
 \begin{equation}\label{up-2}
    \begin{split}
        \frac{1}{2n}\frac{d}{dt}\int_\Omega|\nabla u|^{2n}&=
        \int_\Omega|\nabla u|^{2n-2}\nabla u\cdot\nabla\Delta u-\chi\int_\Omega|\nabla u|^{2n-2}\nabla u\cdot\nabla(\nabla\cdot(u\nabla v))\\
        &\;\;\;+\xi_1\int_\Omega|\nabla u|^{2n-2}\nabla u\cdot\nabla(\nabla\cdot (u\nabla w))\\
        &\;\;\;+\int_\Omega|\nabla u|^{2n-2}\nabla u\cdot\nabla(au-\mu u^{\theta+1})=:\sum_{i=1}^4 J_i.
    \end{split}
\end{equation}
 For $J_1$, noticing the fact that  $2\nabla u\cdot\nabla\Delta u=\Delta|\nabla u|^2-2|D^2u|^2$ and the $L^2$-boundedness of $\nabla u$ in \eqref{gradu-l2-est}, we employ the way to control boundary integral as done in \eqref{reg-v7}  and the GN inequality  to  estimate that
\begin{equation}\label{J1}
    \begin{split}
        J_1&=\frac{1}{2}\int_\Omega|\nabla u|^{2n-2}\Delta|\nabla u|^2-\int_\Omega|\nabla u|^{2n-2}|D^2u|^2\\
        &=\frac{1}{2}\int_{\partial\Omega}|\nabla u|^{2n-2}\frac{\partial|\nabla u|^2}{\partial\nu}-\frac{(n-1)}{2}\int_\Omega|\nabla u|^{2n-4}|\nabla|\nabla u|^2|^2-\int_\Omega|\nabla u|^{2n-2}|D^2u|^2\\
        &\leq -\frac{(n-1)}{3}\int_\Omega|\nabla u|^{2n-4}|\nabla|\nabla u|^2|^2-\frac{1}{2}\int_\Omega|\nabla u|^{2n-2}|D^2u|^2+c_{60}.
    \end{split}
\end{equation}
By the inequality $|\Delta u|\leq\sqrt{n}|D^2u|$ and the boundedness \eqref{hg-e2} and \eqref{delta-v-est}, we  use integration by parts to bound  $J_2$ as follows:
\begin{equation}\label{J2}
\begin{split}
J_2=&\chi\int_\Omega\nabla|\nabla u|^{2n-2}\cdot\nabla u\nabla\cdot(u\nabla v)+\chi\int_\Omega|\nabla u|^{2n-2}\Delta u\nabla\cdot(u\nabla v)\\
=&\chi (n-1)\int_\Omega |\nabla u|^{2(n-2)}\nabla |\nabla u|^2\cdot\nabla u(\nabla u\cdot \nabla v+u\Delta v)\\
&+\chi\int_\Omega|\nabla u|^{2n-2}\Delta u(\nabla u\cdot \nabla v+u\Delta v)\\
\leq& M\chi(n-1)\int_\Omega|\nabla u|^{2n-3}|\nabla|\nabla u|^2|(|\nabla u|+|\Delta v|)\\
        &\;\;\;+M\chi\sqrt{n}\int_\Omega|\nabla u|^{2n-2}|D^2u|(|\nabla u|+|\Delta v|)\\
\leq &\frac{(n-1)}{8}\int_\Omega|\nabla u|^{2n-4}|\nabla|\nabla u|^2|^2+\frac{1}{8}\int_\Omega|\nabla u|^{2n-2}|D^2u|^2\\
&+4(2n-1)M^2\chi^2\int_\Omega|\nabla u|^{2n}+4(2n-1)M^2\chi^2\int_\Omega|\nabla u|^{2n-2}|\Delta v|^2.
\end{split}
\end{equation}
In a similar way, since $|\Delta w|=|u-\bar{u}|\leq 2M$,  and using   Young's inequality, we estimate $J_3$ as follows:
\begin{equation}\label{J3}
\begin{split}
J_3=&-\xi_1\int_\Omega\nabla|\nabla u|^{2n-2}\cdot\nabla u\nabla\cdot(u\nabla w)-\xi_1\int_\Omega|\nabla u|^{2n-2}\Delta u\cdot\nabla(u\nabla w)\\
\leq& M\xi_1(n-1)\int_\Omega|\nabla u|^{2n-3}|\nabla|\nabla u|^2|(|\nabla u|+|\Delta w|)\\
        &\;\;\;+M\xi_1\sqrt{n}\int_\Omega|\nabla u|^{2n-2}|D^2u|(|\nabla u|+|\Delta w|)\\
\leq &\frac{(n-1)}{8}\int_\Omega|\nabla u|^{2n-4}|\nabla|\nabla u|^2|^2+\frac{1}{8}\int_\Omega|\nabla u|^{2n-2}|D^2u|^2\\
&+4(2n-1)M^2\xi_1^2\int_\Omega|\nabla u|^{2n}+16(2n-1)M^4\xi_1^2\int_\Omega|\nabla u|^{2n-2}\\
\leq&\frac{(n-1)}{8}\int_\Omega|\nabla u|^{2n-4}|\nabla|\nabla u|^2|^2+\frac{1}{8}\int_\Omega|\nabla u|^{2n-2}|D^2u|^2\\
&+2(4n-1)M^2\xi_1^2\int_\Omega|\nabla u|^{2n}+c_{61}.
\end{split}
\end{equation}
At last, the term $J_4$ can be easily computed:
\begin{equation}\label{J4}
J_4= a\int_\Omega |\nabla u|^{2n}-\mu(\theta+1)\int_\Omega u^\theta |\nabla u|^{2n}.
\end{equation}
Substituting \eqref{J1}, \eqref{J2}, \eqref{J3} and \eqref{J4} into \eqref{up-2} and using Young's inequality, one has a differential  inequality of the form:
\begin{equation}\label{up-3}
\begin{split}
&\frac{d}{dt}\int_\Omega|\nabla u|^{2n}+2n\int_\Omega|\nabla u|^{2n}+\frac{(n-1)n}{6}\int_\Omega|\nabla u|^{2n-4}|\nabla|\nabla u|^2|^2\\
&\ \ \ +\frac{n}{2}\int_\Omega|\nabla u|^{2n-2}|D^2u|^2\\
&\leq c_{62} \int_\Omega |\nabla u|^{2n}+c_{62}\int_\Omega|\nabla u|^{2n-2}|\Delta v|^2+c_{62}\\
&\leq c_{63} \int_\Omega |\nabla u|^{2n}+c_{63}\int_\Omega |\Delta v|^{2n}+c_{63}.
\end{split}
\end{equation}
Since  $\|u(\cdot,t)\|_{L^\infty}\leq M$ in  \eqref{hg-e2},  the Young's inequality enables one to deduce that
\begin{align*}
    2c_{63}\int_\Omega|\nabla u|^{2n}&=2c_{63}\left(-(n-1)\int_\Omega u|\nabla u|^{2n-4}\nabla|\nabla u|^2\cdot\nabla u-\int_\Omega u|\nabla u|^{2n-2}\Delta u\right)\\
    &\leq 2c_{63}M\left((n-1)\int_\Omega |\nabla u|^{2n-3}|\nabla|\nabla u|^2|+\sqrt{n}\int_\Omega |\nabla u|^{2n-2}|D^2u|\right)\\
    &\leq\frac{n(n-1)}{6}\int_\Omega |\nabla u|^{2n-4}|\nabla|\nabla u|^2|^2+\frac{n}{2}\int_\Omega|\nabla u|^{2n-2}|D^2u|^2+
    c_{64}\int_\Omega|\nabla u|^{2n-2}\\
    &\leq\frac{n(n-1)}{6}\int_\Omega |\nabla u|^{2n-4}|\nabla|\nabla u|^2|^2+\frac{n}{2}\int_\Omega|\nabla u|^{2n-2}|D^2u|^2\\
    &\;\;\;+c_{63}\int_\Omega|\nabla u|^{2n}+c_{65},
\end{align*}
which, upon being substituted  into \eqref{up-3},  allows us to conclude
\begin{align*}
 \frac{d}{dt}\int_\Omega|\nabla u|^{2n}+2n\int_\Omega|\nabla u|^{2n}\leq c_{63}\int_\Omega|\Delta v|^{2n}+c_{66}.
\end{align*}
Solving this simple differential inequality via integrating factor method, we get
\begin{equation}\label{up-8}
       \int_\Omega|\nabla u(\cdot,t)|^{2n}\leq \frac{c_{66}}{2n}+\int_\Omega|\nabla u(\cdot,1)|^{2n} +c_{63}e^{-2nt}\int_1^t e^{2ns}\int_\Omega|\Delta v(\cdot,s)|^{2n}ds, \quad  \forall t\geq1.
\end{equation}
 Next, we apply the widely known maximal Sobolev regularity to the $v$-equation to bound the  third term on the right.  To this end, with the help of  the $w$-equation, we first rewrite the $v$-equation  in  \eqref{system}  in the following form:
\begin{equation}\label{up-9}
\begin{split}
v_t&=d\Delta v-v+\xi_2\nabla v\cdot\nabla w+\xi_2 v(\bar{u}-u)+u.\\
\end{split}
\end{equation}
Let  $h(x,t):=\xi_2\nabla v\cdot\nabla w+\xi_2 v(\bar{u}-u)+u$. Then it follows from the boundedness  \eqref{hg-e2} that
\begin{equation}\label{up-10}
|h|\leq M^2\xi_2+2M^2\xi_2+M=3M^2\xi_2+M \text{  on   } \Omega \times (0, \infty).
\end{equation}
Now, an application of the  maximal Sobolev regularity to \eqref{up-9} with \eqref{up-10} shows that
\begin{equation*}
\begin{split}
    \int_1^te^{2ns}\int_\Omega|\Delta v(\cdot, t)|^{2n}
    &\leq c_{67}\left(e^{2n}\int_\Omega|\Delta v(\cdot,1)|^{2n}+\int_1^te^{2ns}\int_\Omega|h|^{2n}\right)\\
    &\leq  c_{67}\left(e^{2n}\int_\Omega|\Delta v(\cdot,1)|^{2n}+\frac{(3M^2\xi_2+M)^{2n}|\Omega|}{2n}e^{2nt}\right),
\end{split}
\end{equation*}
which, substituted into \eqref{up-8}, gives, for $\forall t\geq 1$,
\begin{equation*}
       \int_\Omega|\nabla u(\cdot,t)|^{2n}
       \leq \frac{c_{66}}{2n}+\int_\Omega|\nabla u(\cdot,1)|^{2n} +c_{63}c_{67}\left(\int_\Omega|\Delta v(\cdot,1)|^{2n}+\frac{(3M^2\xi_2+M)^{2n}|\Omega|}{2n}\right),
\end{equation*}
 yielding  our desired gradient estimate \eqref{up-1}.
\end{proof}

\section{Global stability: Proof of Theorem \ref{LT}}
Given the enhanced regularity properties in the preceding section,   we shall examine  the long time dynamics of bounded solutions as obtained in Theorem \ref{GB}. Under certain conditions, we shall show convergence and exponential convergence  of bounded solutions  to the unique constant steady state as  time  goes to infinity.

\subsection{Case 1: $a=\mu=0$} In this case, the system  \eqref{system} becomes
\begin{equation}\label{1-1}
    \begin{cases}
    u_t=\Delta u-\chi\nabla\cdot(u\nabla v)+\xi_1\nabla\cdot(u\nabla w), &x\in\Omega, t>0,\\
    v_t=d\Delta v+\xi_2\nabla\cdot(v\nabla w)+u-v, &x\in\Omega, t>0,\\
    0=\Delta w+u-\bar{u}_0, \  \   \int_\Omega w=0, & x\in\Omega, t>0,\\
    \frac{\partial u}{\partial\nu}=\frac{\partial v}{\partial\nu}=\frac{\partial w}{\partial\nu}=0, &x\in\partial\Omega, t>0,\\
    u(x,0)=u_0(x),~v(x,0)=v_0(x),&x\in\Omega.
    \end{cases}
\end{equation}

\begin{lemma}\label{GS}
 The  global classical solution of \eqref{1-1} satisfies  the following identity:
\begin{equation}\label{GS-1}
\begin{split}
&\frac{d}{dt}\left(\int_\Omega u\ln \frac{u}{\bar{u}}+\frac{\chi}{2}\int_\Omega |\nabla v|^2\right)+\int_\Omega \frac{|\nabla u|^2}{u}+d\chi\int_\Omega |\Delta v|^2+\chi \int_\Omega |\nabla v|^2+\xi_1\int_\Omega |\Delta w|^2\\
&=2\chi\int_\Omega \Delta w\cdot \Delta v-\chi\xi_2\int_\Omega \nabla \cdot(v\nabla w)\Delta v.
\end{split}
\end{equation}
\end{lemma}
\begin{proof}
Multiplying the first equation of \eqref{1-1} by $\ln u-\ln \bar{u}$, we have
\begin{equation}\label{GS-2}
\frac{d}{dt}\int_\Omega u\ln \frac{u}{\bar{u}}+\int_\Omega \frac{|\nabla u|^2}{u}=\chi\int_\Omega \nabla u\cdot \nabla v-\xi_1\int_\Omega \nabla u\cdot \nabla w.
\end{equation}
We multiply the second equation of \eqref{1-1} by $-\Delta v$ to obtain
\begin{equation}\label{GS-3}
\frac{1}{2}\frac{d}{dt}\int_\Omega |\nabla v|^2+d\int_\Omega |\Delta v|^2+\int_\Omega |\nabla v|^2=\int_\Omega \nabla u\cdot \nabla v-\xi_2\int_\Omega \nabla \cdot(v\nabla w)\Delta v.
\end{equation}
Then multiplying \eqref{GS-3} by $\chi$ and adding it to \eqref{GS-2}, and using the facts $u=\bar{u}_0-\Delta w$ and $\int_\Omega \Delta v=0=\int_\Omega \Delta w$ due to homogeneous boundary conditions, we obtain
\begin{equation*}
\begin{split}
&\frac{d}{dt}\left(\int_\Omega u\ln \frac{u}{\bar{u}}+\frac{\chi}{2}\int_\Omega |\nabla v|^2\right)+\int_\Omega \frac{|\nabla u|^2}{u}+d\chi\int_\Omega |\Delta v|^2+\chi\int_\Omega |\nabla v|^2\\
&=2\chi \int_\Omega \nabla u\cdot \nabla v-\xi_1\int_\Omega \nabla u\cdot \nabla w-\chi\xi_2\int_\Omega \nabla \cdot(v\nabla w)\Delta v\\
&=2\chi \int_\Omega \Delta w\cdot\Delta v-\xi_1\int_\Omega |\Delta w|^2-\chi\xi_2\int_\Omega \nabla \cdot(v\nabla w)\Delta v,
\end{split}
\end{equation*}
which gives rise to \eqref{GS-1}.
\end{proof}
In the limiting case of $\chi=0$, there is no bad terms to be estimated by \eqref{GS-1}, we can directly jump to \eqref{GS-10} below,  and hence its proof is simpler. We  proceed to treat the hard case of $\chi>0$, for that purpose,  we denote
\begin{equation}\label{AB-def}
A=A(d)=\sup_{(x,t)\in\Omega\times(0, \infty)}v(x,t), \quad \quad \quad  B=B(d)=\sup_{(x,t)\in\Omega\times(0, \infty)}|\nabla w(x,t)|.
\end{equation}
Then by the qualitative bounds for $v$ and $\nabla w$ in \eqref{v-infty-bdd} and \eqref{gradw-infty-bdd} with \eqref{m1c-def} and \eqref{d-omega}, we see  that they are bounded for large $d$ and hence
\begin{equation}\label{d*-def}
d_0(\chi)=\inf\left\{\hat{d}_0>0: \   \left(d- \frac{\left(2+A\xi_2\right)^2\chi}{4\xi_1}- \frac{B^2\xi_2^2}{4}\right)\chi \geq 0 \text{ for all }  d\geq \hat{d}_0\right\}<+\infty.
\end{equation}
This allows us to include the  the limiting case of $\chi=0$, which corresponds to $d_0(0)=0$.
\begin{lemma}\label{GS-4}
Assume $\xi_1\geq \xi_0\chi^2$ and $d\geq d_0(\chi)$. Then the global classical  solution $(u,v,w)$ of \eqref{1-1} satisfies the following estimates:
\begin{equation}\label{GS-5}
\int_1^t\int_\Omega \frac{|\nabla u|^2}{u}\leq C_9, \quad \quad \forall t>1.
\end{equation}
In addition, if $d>d_0(\chi)$, then there exists $\zeta_1>0$ such that
\begin{equation}\label{GS-5*}
\|u(\cdot, t)-\bar{u}_0\|_{L^1}+\|\nabla v(\cdot, t)\|_{L^2}\leq C_{10}e^{-\zeta_1 t}, \quad \quad \forall t>0.
\end{equation}
\end{lemma}
\begin{proof}
Using the definitions of $A$ and $B$ in  \eqref{AB-def}, we infer from  \eqref{GS-1}  of Lemma \ref{GS} that
\begin{equation}\label{GS-6}
\begin{split}
&\frac{d}{dt}\left(\int_\Omega u\ln \frac{u}{\bar{u}}+\frac{\chi}{2}\int_\Omega |\nabla v|^2\right)+\int_\Omega \frac{|\nabla u|^2}{u}+d\chi\int_\Omega |\Delta v|^2+\chi \int_\Omega |\nabla v|^2+\xi_1\int_\Omega |\Delta w|^2\\
&=2\chi\int_\Omega \Delta w\cdot \Delta v-\chi\xi_2\int_\Omega v\Delta w \cdot\Delta v-\xi_2\chi\int_\Omega \nabla v\cdot \nabla w \cdot\Delta v\\
&\leq \left(2+A  \xi_2\right)\chi\int_\Omega |\Delta w||\Delta v|+B\chi\xi_2\int_\Omega |\nabla v||\Delta v|.
\end{split}
\end{equation}
Using the Young's inequality, we readily derive
\begin{equation}\label{GS-7}
\left(2+A  \xi_2\right)\chi\int_\Omega |\Delta w||\Delta v|\leq \xi_1\int_\Omega |\Delta w|^2+\frac{(2+A\xi_2)^2\chi^2}{4\xi_1}\int_\Omega|\Delta v|^2,
\end{equation}
and
\begin{equation}\label{GS-8}
B\chi \xi_2\int_\Omega |\nabla v||\Delta v|\leq (1-\varepsilon_1)\chi \int_\Omega|\nabla v|^2+\frac{B^2\chi\xi_2^2}{4(1-\varepsilon_1)}\int_\Omega |\Delta v|^2,
\end{equation}
where, due to $d\geq d_0(\chi)$ with $d_0(\chi)$ defined by \eqref{d*-def}, $\varepsilon_1\in [ 0, 1)$ is defined by
\begin{equation}\label{dv}
\varepsilon_1=\begin{cases}
0, &\mathrm{if}\ \ \ d=d_0(\chi),\\
\frac{1}{2}\left(d- \frac{\left(2+A\xi_2\right)^2\chi}{4\xi_1}- \frac{B^2\xi_2^2}{4}\right)\left(d- \frac{\left(2+A\xi_2\right)^2\chi}{4\xi_1}\right)^{-1}, &\mathrm{if}\ \ \ d>d_0(\chi).\\
\end{cases}
\end{equation}
A substitution of \eqref{GS-7} and \eqref{GS-8} into \eqref{GS-6} gives
\begin{equation}\label{GS-9}
\begin{split}
&\frac{d}{dt}\left(\int_\Omega u\ln \frac{u}{\bar{u}}+\frac{\chi}{2}\int_\Omega |\nabla v|^2\right)+\int_\Omega \frac{|\nabla u|^2}{u}+\varepsilon_1\chi\int_\Omega|\nabla v|^2\\[0.25cm]
&+\left(d-\frac{(2+A\xi_2 )^2\chi}{4\xi_1}-\frac{B^2\xi_2^2}{4(1-\varepsilon_1)}\right)\chi\int_\Omega |\Delta v|^2\leq 0.
\end{split}
\end{equation}
Since  $d\geq d_0(\chi)$, using the definitions of $d_0$ and  $\varepsilon_1$ in \eqref{d*-def} and \eqref{dv}, one has
\begin{equation*}
\left(d-\frac{(2+A\xi_2 )^2\chi}{4\xi_1}-\frac{B^2\xi_2^2}{4(1-\varepsilon_1)}\right)\chi\geq 0,
\end{equation*}
which, substituted into \eqref{GS-9}, yields
\begin{equation}\label{GS-9**}
\frac{d}{dt}\left(\int_\Omega u\ln \frac{u}{\bar{u}}+\frac{\chi}{2}\int_\Omega |\nabla v|^2\right)+\int_\Omega \frac{|\nabla u|^2}{u}+\varepsilon_1\chi\int_\Omega|\nabla v|^2\leq 0.
\end{equation}
Then an integration  with respect to $t$  over $(1,t)$ shows that
\begin{equation}\label{GS-10}
\int_\Omega u\ln \frac{u}{\bar{u}}+\frac{\chi}{2}\int_\Omega |\nabla v|^2+\int_1^t\int_\Omega \frac{|\nabla u|^2}{u}\leq \int_\Omega u(\cdot, 1)\ln \frac{u(\cdot, 1)}{\bar{u}_0}+\frac{\chi}{2}\int_\Omega |\nabla v(\cdot,1)|^2.
\end{equation}
A use of  Lemma \ref{II} along with the fact $\bar{u}=\bar{u}_0$ entails
\begin{equation}\label{GS-11*}
\int_\Omega u\ln \frac{u}{\bar{u}}\geq \frac{1}{2\bar{u}_0}\|u-\bar{u}_0\|_{L^1}^2\geq 0,
\end{equation}
and
\begin{equation}\label{GS-12}
\int_\Omega u(\cdot, 1)\ln \frac{u(\cdot, 1)}{\bar{u}_0}\leq \frac{1}{\bar{u}_0}\|u(\cdot,1)-\bar{u}_0\|_{L^2}^2.
\end{equation}
Substituting \eqref{GS-11*} and \eqref{GS-12} into \eqref{GS-10}, we obtain \eqref{GS-5} directly.

On the other hand, using Lemma \ref{II} again and noting the fact of $\|u\|_{L^\infty}\leq M$, we use the Poincar\'{e} inequality to derive that
\begin{equation*}\label{GS-11}
\int_\Omega u\ln \frac{u}{\bar{u}}\leq \frac{1}{\bar{u}}\|u-\bar{u}\|_{L^2}^2\leq \frac{c_{68}}{\bar{u}_0}\|\nabla u\|_{L^2}^2\leq \frac{c_{68}M}{\bar{u}_0}\int_\Omega \frac{|\nabla u|^2}{u},
\end{equation*}
which yields
\begin{equation}\label{GS-13}
\frac{\bar{u}_0}{c_{68}M}\int_\Omega u\ln \frac{u}{\bar{u}}\leq \int_\Omega \frac{|\nabla u|^2}{u}.
\end{equation}
Substituting \eqref{GS-13} into \eqref{GS-9**}, and recalling  the fact $\varepsilon_1>0$ due to $d>d_0(\chi)$, then  we find a positive constant $c_{69}:=\min\{\frac{\bar{u}_0}{c_{68}M },2\varepsilon_1\}$ such that
\begin{equation*}\label{GS-9*}
\begin{split}
&\frac{d}{dt}\left(\int_\Omega u\ln \frac{u}{\bar{u}}+\frac{\chi}{2}\int_\Omega |\nabla v|^2\right)+c_{69}\left(\int_\Omega u\ln \frac{u}{\bar{u}}+\frac{\chi}{2}\int_\Omega |\nabla v|^2\right)\leq 0,
\end{split}
\end{equation*}
which immediately entails, for $t\geq 1$,
\begin{equation}\label{GS-14}
\int_\Omega u\ln \frac{u}{\bar{u}}+\frac{\chi}{2}\int_\Omega |\nabla v|^2\leq \left(\int_\Omega u(\cdot,1)\ln \frac{u(\cdot,1)}{\bar{u}_0}+\frac{\chi}{2}\int_\Omega |\nabla v(\cdot,1)|^2\right)e^{-c_{69}(t-1)}.
\end{equation}
Then the combination of \eqref{GS-11*}, \eqref{GS-12} and \eqref{GS-14} implies \eqref{GS-5*}.
\end{proof}

\begin{lemma}\label{K*}
Under Lemma \ref{GS-4}, the $u,w$-components of the global bounded  classical  solution   of \eqref{1-1} fulfills the following properties:
 \begin{itemize}
 \item[(uc1)] If $d\geq d_0(\chi)$, then $(u, w)$ decays to $(\bar{u}_0,0)$ uniformly:
\begin{equation}\label{K1}
\|u(\cdot,t)-\bar{u}_0\|_{L^\infty}+\|w(\cdot,t)\|_{W^{2,\infty}}\to 0 \ \ \mathrm{as}\ \  t\to\infty.
\end{equation}
\item[(uc2)] If $d>d_0(\chi)$, then $(u,w)$ decays exponentially to $(\bar{u}_0, 0)$: for some $\zeta_2>0$,
\begin{equation}\label{K1*}
\|u(\cdot,t)-\bar{u}_0\|_{L^\infty}+\|w(\cdot,t)\|_{W^{2,\infty}}\leq C_{11}e^{-\zeta_2 t},  \ \ \ \ \forall t>0.
\end{equation}
\end{itemize}
\end{lemma}
\begin{proof}
 Notice from  Theorem \ref{GB}  that $\|u(\cdot,t)\|_{L^\infty}$ is uniformly bounded, and thus the space-time estimate  \eqref{GS-5} along with  Poincar\'{e} inequality ensures
\begin{equation}\label{K4}
\int_1^\infty \int_\Omega |u-\bar{u}_0|^2\leq c_{69}\int_1^\infty\int_\Omega |\nabla u|^2\leq c_{70}.
\end{equation}
Then the uniform continuity of $\|u-\bar{u}_0\|_{L^2}^2$ implied by \eqref{hg-e1} shows that
\begin{equation}\label{l2-conv}
\|u-\bar{u}_0\|_{L^2}\rightarrow 0 \text{ as } t\rightarrow \infty.
\end{equation}
 Thus, with the boundedness of $\|\nabla u\|_{L^{2n}}$ in \eqref{up-1},  the Gagliardo-Nirenberg inequality gives
\begin{equation}\label{u-conver}
    \begin{split}
        \|u(\cdot,t)-\bar{u}_0\|_{L^\infty}\leq& c_{71}\left(\|\nabla u(\cdot,t)\|_{L^{2n}}^{\frac{n}{n+1}}\|u(\cdot,t)-\bar{u}_0\|_{L^2}^{\frac{1}{n+1}}
        +\|u(\cdot,t)-\bar{u}_0\|_{L^2}\right)\\
        \leq & c_{72} \left(\|u(\cdot,t)-\bar{u}_0\|_{L^2}^{\frac{1}{n+1}}
        +\|u(\cdot,t)-\bar{u}_0\|_{L^2}\right) \rightarrow 0  \text{ as } t\rightarrow \infty.
    \end{split}
    \end{equation}
This establishes the $u$-convergence in \eqref{K1}. We here also provide an alterative short proof based on   \cite[Lemma 3.10]{Tao-Winkler-SIMA-2015}. Indeed, assume to the contrary, then there would exist some sequences $(x_j)_{j\in\mathbb{N}}\subset\Omega$  and $(t_j)_{j\in\mathbb{N}}\subset(0,\infty)$  satisfying $t_j\to\infty$ as $j\to\infty$ such that
\begin{equation*}
|u(x_j,t_j)-\bar{u}_0|\geq c_{73}, \ \ \ \forall j\in\mathbb{N}.
\end{equation*}
The uniform continuity of $u$ due to Lemma \ref{hg} warrants there exist $r>0$ and $\delta>0$ such  that,  for any $j\in\mathbb{N}$,
\begin{equation}\label{K6}
|u-\bar{u}_0|\geq \frac{c_{73}}{2} \text{ on }  B_r(x_j)\cap \Omega \times (t_j,t_j+\delta).
\end{equation}
The smoothness of $\partial\Omega$ shows  that
\begin{equation}\label{K7}
|B_r(x_j)\cap\Omega|\geq c_{74}, \ \ \forall  x_j\in \Omega.
\end{equation}
Therefore, for all $j\in \mathbb{N}$, it follows from  \eqref{K6} and \eqref{K7} that
\begin{equation}\label{K8}
\int_{t_j}^{t_j+\delta}\int_\Omega |u-\bar{u}_0|^2  \geq \int_{t_j}^{t_j+\delta}\int_{B_r(x_j)\cap\Omega} |u-\bar{u}_0|^2\geq \frac{c_{73}^2c_{74}\delta}{4},
\end{equation}
which clearly  contradicts the following fact due to \eqref{K4}:
\begin{equation*}
\int_{t_j}^{t_j+\delta}\int_\Omega |u-\bar{u}_0|^2\leq \int_{t_j}^\infty\int_\Omega |u-\bar{u}_0|^2\to  0 \  \mathrm{ as }\ \ j\to\infty.
\end{equation*}
This contradiction gives rise to \eqref{K1}.

In the case of $d>d_0(\chi)$, we shall apply Lemma \ref{up} and \eqref{GS-5*} to show the $L^\infty$- exponential decay \eqref{K1*}. Indeed, employing  the Gagliardo-Nirenberg inequality, the $L^1$-exponential decay \eqref{GS-5*} and the boundedness of $\|\nabla u\|_{L^{2n}}$ in Lemma \ref{up}, we conclude that
\begin{equation*}
    \begin{split}
        \|u(\cdot,t)-\bar{u}_0\|_{L^\infty}\leq& c_{75}\left(\|\nabla u(\cdot,t)\|_{L^{2n}}^{\frac{2n}{2n+1}}\|u(\cdot,t)-\bar{u}_0\|_{L^1}^{\frac{1}{2n+1}}
        +\|u(\cdot,t)-\bar{u}_0\|_{L^1}\right)\\
        \leq & c_{76}e^{-\frac{\zeta_1}{2n+1}t},
    \end{split}
\end{equation*}
which shows  our $u$-exponential decay estimate  \eqref{K1}.

Finally,  we apply the $W^{2,\infty}$-estimate to \eqref{w-2p} obtain
$$
\|w(\cdot, t)\|_{W^{2,\infty}}\leq c_{77}\|u(\cdot, t)-\bar{u}_0\|_{L^\infty},
$$
hence, the assertions for the convergence for $w$ follow directly the convergence of $u$.
\end{proof}
\begin{lemma}\label{vc}
 Under Lemma \ref{GS-4}, the $v$-component of the global bounded  classical  solution   of \eqref{1-1} enjoys  the following convergence properties:
 \begin{itemize}
 \item[(vc1)] If $d\geq d_0(\chi)$, then $v$ decays to $\bar{u}_0$ uniformly:
\begin{equation}\label{vc-1}
\|v(\cdot,t)-\bar{u}_0\|_{L^\infty}\to 0 \ \ \mathrm{as}\ \  t\to\infty.
\end{equation}
\item[(vc2)]  If $d>d_0(\chi)$, then $v$ decays exponentially to $\bar{u}_0$ uniformly: for some $\zeta_3>0$,
\begin{equation}\label{vc2}
\|v(\cdot,t)-\bar{u}_0\|_{L^\infty}\leq C_{12}e^{-\zeta_3 t},  \ \ \ \ \forall t>0.
\end{equation}
\end{itemize}
\end{lemma}
\begin{proof}
We first  rewrite the $v$-equation in \eqref{1-1} as
\begin{equation}\label{vc-2}
(v-\bar{u}_0)_t=d\Delta (v-\bar{u}_0)+\xi_2\nabla\cdot((v-\bar{u}_0)\nabla w)+\xi_2 \bar{u}_0\Delta w+u-\bar{u}_0-(v-\bar{u}_0).
\end{equation}
Multiplying \eqref{vc-2} by  $v-\bar{u}_0$, integrating over $\Omega$  by parts and using the fact $\Delta w=\bar{u}_0-u$  by  \eqref{1-1}, we compute that
\begin{equation}\label{vc-3}
\begin{split}
&\frac{1}{2}\frac{d}{dt}\int_\Omega (v-\bar{u}_0)^2+d\int_\Omega |\nabla v|^2+\int_\Omega (v-\bar{u}_0)^2\\
&=\frac{\xi_2}{2}\int_\Omega(v-\bar{u}_0)^2\Delta w+\xi_2\bar{u}_0\int_\Omega (v-\bar{u}_0)\Delta w+\int_\Omega(u-\bar{u}_0)(v-\bar{u}_0)\\
&=\frac{\xi_2}{2}\int_\Omega(v-\bar{u}_0)^2(\bar{u}_0-u)+\xi_2\bar{u}_0\int_\Omega (v-\bar{u}_0)(\bar{u}_0-u)+\int_\Omega(u-\bar{u}_0)(v-\bar{u}_0).
\end{split}
\end{equation}
The  fact  $\|u-\bar{u}_0\|_{L^\infty}\to 0$ as $t\to\infty$ by \eqref{K1} shows   there exists $t_1>0$ such that
\begin{equation}\label{vc-4}
 \|u(\cdot, t)-\bar{u}_0\|_{L^\infty} \leq \frac{1}{2\xi_2}, \quad \quad  t\geq t_1.
\end{equation}
On the other hand, using Young's inequality, we derive
\begin{equation}\label{vc-5}
\xi_2\bar{u}_0\int_\Omega (v-\bar{u}_0)(\bar{u}_0-u)\leq \frac{1}{8}\int_\Omega (v-\bar{u}_0)^2+2\xi_2^2\bar{u}_0^2 \int_\Omega(u-\bar{u}_0)^2,
\end{equation}
and
\begin{equation}\label{vc-6}
\int_\Omega(u-\bar{u}_0)(v-\bar{u}_0)\leq \frac{1}{8}\int_\Omega (v-\bar{u}_0)^2+2\int_\Omega(u-\bar{u}_0)^2.
\end{equation}
Combining  \eqref{vc-4},\eqref{vc-5},  \eqref{vc-6} and  \eqref{vc-3}, we conclude,  for all $t\geq t_1$,  that
$$
\frac{d}{dt}\int_\Omega (v-\bar{u}_0)^2+2d\int_\Omega |\nabla v|^2+\int_\Omega (v-\bar{u}_0)^2
\leq 4(1+\xi_2^2\bar{u}_0^2)\int_\Omega(u-\bar{u}_0)^2,
$$
which, upon being solved and  the fact \eqref{l2-conv},  implies
\begin{equation}\label{vc-8}
\begin{split}
\|v(\cdot,t)-\bar{u}_0\|_{L^2}^2&\leq \|v(\cdot,1)-\bar{u}_0\|_{L^2}^2e^{-(t-1)}+4(1+\xi_2^2\bar{u}_0^2)
\int_1^te^{-(t-s)}\int_\Omega(u(\cdot, s)-\bar{u}_0)^2\\
&\rightarrow 0 \ \ \text{ as } t\rightarrow \infty.
\end{split}
\end{equation}
Then, with the $L^2$-convergence \eqref{vc-8} at hand, replacing $v$ with $u$ in the Gagliardo-Nirenberg inequality  \eqref{u-conver}, we readily obtain \eqref{vc-1}.

In the case of $d>d_0(\chi)$, to show  $v$ decays exponentially, we first see from \eqref{1-1}  that
\begin{equation}\label{vbar-con}
\bar{v}-\bar{u}_0=\left(\bar{v}_0-\bar{u}_0\right)e^{-t}, \ \ \ \forall t>0.
\end{equation}
Now, we use  Poincar\'{e} inequality  and the exponential decays  \eqref{GS-5*} and \eqref{vbar-con} to deduce
\begin{equation*}
\begin{split}
\|v(\cdot,t)-\bar{u}_0\|_{L^2}&\leq \|v-\bar{v}\|_{L^2}+\|\bar{v}-\bar{u}_0\|_{L^2}\\
&\leq c_{78}\|\nabla v\|_{L^2}+|\bar{v}_0-\bar{u}_0||\Omega|e^{-t}\\
&\leq c_{79}e^{-\min\{\zeta_1, \ 1\}t}, \ \ \forall t>0.
\end{split}
\end{equation*}
Now, since $\|\nabla v\|_{L^{2n}}$ is uniformly-in-time bounded by Theorem \ref{GB}, replacing $v$ with $u$ in the Gagliardo-Nirenberg inequality  \eqref{u-conver}, we readily improve this $L^2$-exponential decay to $L^\infty$-exponential decay of $v$ as in  \eqref{vc2}.
\end{proof}

\subsection{Case 2:  $a, \mu>0$}
In this subsection, we shall study the large time behavior of solution for the complete system \eqref{system}  with  $a, \mu>0$.  For convenience,  we rewrite it here:
\begin{equation}\label{1-2}
    \begin{cases}
    u_t=\Delta u-\chi\nabla\cdot(u\nabla v)+\xi_1\nabla\cdot(u\nabla w)+u(a-\mu u^\theta), &x\in\Omega, t>0,\\
    v_t=d\Delta v+\xi_2\nabla\cdot(v\nabla w)+u-v, &x\in\Omega, t>0,\\
    0=\Delta w+u-\bar{u}, \  \   \int_\Omega w=0, & x\in\Omega, t>0,\\
    \frac{\partial u}{\partial\nu}=\frac{\partial v}{\partial\nu}=\frac{\partial w}{\partial\nu}=0, &x\in\partial\Omega, t>0,\\
    u(x,0)=u_0(x),~v(x,0)=v_0(x),&x\in\Omega.
    \end{cases}
\end{equation}
\begin{lemma}\label{lg-w}
Let $C_p$ denote the Poincar\'{e} constant defined by \eqref{poincare-cont} and $(u,v,w)$ be the solution of the system \eqref{1-2}. Then
\begin{equation}\label{lg-w1}
  \|\nabla w\|_{L^2}\leq C_p\|u-b\|_{L^2}, \ \  \|\Delta w\|_{L^2}\leq  \|u-b\|_{L^2}, \ \   \| w\|_{W^{2,2}}\leq C_{13}\|u-b\|_{L^2}, \ \ b=(\frac{a}{\mu})^\frac{1}{\theta}.
\end{equation}
\end{lemma}
\begin{proof} By the $w$-equation in \eqref{1-1}  and the Poincar\'{e} inequality  due to $\int_\Omega  w=0$, we infer
\begin{equation*}
\int_\Omega |\nabla w|^2=\int_\Omega (u-b)w+(b-\bar{u})\int_\Omega w\leq \|u-b\|_{L^2}\|w\|_{L^2}\leq C_p\|u-b\|_{L^2}\|\nabla w\|_{L^2}.
\end{equation*}
which shows
\begin{equation}\label{wl2-con}
\|\nabla w\|_{L^2}\leq C_p\|u-b\|_{L^2}.
\end{equation}
Next, we deduce from the third equation in \eqref{1-1}  and the fact that $\int_\Omega \Delta w=0$ that
\begin{equation*}
\begin{split}
\int_\Omega |\Delta w|^2&=\int_\Omega (b-u)\Delta w+(\bar{u}-b)\int_\Omega \Delta w\\
&=-\int_\Omega (u-b)\Delta w\leq \left(\int_\Omega (u-b)^2\right)^\frac{1}{2}\left(\int_\Omega |\Delta w|^2\right)^\frac{1}{2},
\end{split}
\end{equation*}
which  along with the $W^{2,2}$-elliptic estimate applied to \eqref{w-2p} yields
\begin{equation*}
\|  w\|_{W^{2,2}}\leq c_{80}\|\Delta w\|_{L^2}\leq c_{80}\|u-b\|_{L^2}.
\end{equation*}
This together with \eqref{wl2-con} gives  rise to  \eqref{lg-w1}.
\end{proof}

\begin{lemma}
The classical  solution of  \eqref{1-2} satisfies
\begin{equation}\label{lg-u1}
 \frac{d}{dt}\int_\Omega (u-b-b\ln \frac{u}{b})+\mu\int_\Omega (u-b)(u^\theta-b^\theta)\leq  \frac{b\chi^2}{2}\int_\Omega |\nabla v|^2+\frac{b\xi_1^2C_p^2}{2}\int_\Omega (u-b)^2.
\end{equation}
\end{lemma}
\begin{proof}
We use integration by parts to deduce from   \eqref{1-2} that
\begin{equation}\label{lg-1}
\begin{split}
\frac{d}{dt}\int_\Omega (u-b-b\ln \frac{u}{b})&+b\int_\Omega \frac{|\nabla u|^2}{u^2}+\mu\int_\Omega (u-b)(u^\theta-b^\theta)\\
&= b\chi\int_\Omega \frac{\nabla u\cdot\nabla v}{u} - b\xi_1\int_\Omega \frac{\nabla u\cdot\nabla w}{u}.
\end{split}
\end{equation}
Using Young's inequality, one has
\begin{equation}\label{lg-2}
b\chi\int_\Omega \frac{\nabla u\cdot\nabla v}{u}\leq \frac{b}{2}\int_\Omega \frac{|\nabla u|^2}{u^2}+ \frac{b\chi^2}{2}\int_\Omega |\nabla v|^2,
\end{equation}
and, due to  \eqref{lg-w1} of Lemma \ref{lg-w},
\begin{equation}\label{lg-3}
-b\xi_1\int_\Omega \frac{\nabla u\cdot\nabla w}{u}\leq \frac{b}{2}\int_\Omega \frac{|\nabla u|^2}{u^2}+ \frac{b\xi_1^2}{2}\int_\Omega |\nabla w|^2\leq \frac{b}{2}\int_\Omega \frac{|\nabla u|^2}{u^2}+\frac{b\xi_1^2C_p^2}{2}\int_\Omega (u-b)^2.
\end{equation}
Then we substitute \eqref{lg-2} and \eqref{lg-3}  into \eqref{lg-1} to obtain \eqref{lg-u1}.
\end{proof}
\begin{lemma}
The global bounded  classical solution of \eqref{1-2} fulfils
\begin{equation}\label{l2v-dev-p}
\frac{d}{dt}\int_\Omega(v-b)^2+\int_\Omega(v-b)^2+d\int_\Omega|\nabla v|^2\leq \left(1+\frac{\xi_2^2M_0^2C_p^2}{d}\right)\int_\Omega (u-b)^2.
\end{equation}
\begin{proof} From  \eqref{1-2}, we use integration by parts,  the fact $\|v(\cdot,t)\|_{L^\infty}\leq M_0$ in \eqref{b-v1}, the H\"older inequality and  Young's inequality  to estimate
\begin{equation*}
\begin{split}
&\frac{d}{dt}\int_\Omega(v-b)^2+2\int_\Omega(v-b)^2+2d\int_\Omega|\nabla v|^2\\
&=-2\xi_2\int_\Omega v\nabla v\cdot\nabla w+2\int_\Omega (u-b)(v-b)\\
&\leq 2\xi_2M_0\|\nabla v\|_{L^2}\|\nabla w\|_{L^2} + \int_\Omega(u-b)^2+\int_\Omega(v-b)^2\\
& \leq d\int_\Omega|\nabla v|^2 + \frac{\xi_2^2M_0^2}{d}\int_\Omega|\nabla w|^2 + \int_\Omega(u-b)^2+\int_\Omega(v-b)^2,
\end{split}
\end{equation*}
which in conjunction with  \eqref{lg-w1} shows \eqref{l2v-dev-p}.
\end{proof}
\end{lemma}

We are now ready to show the exponential decay of global bounded solutions.
\begin{lemma}\label{cv-u}
Let $(u,v,w)$ be the global  solution  of \eqref{1-2}  obtained in Theorem \ref{GB}. If $\mu>0$ satisfies \eqref{mu-large-lt-intro}, then $(u,v, w)$ decays exponentially to $(b,b, 0)$: for some $\sigma>0$,
\begin{equation}\label{cv-uvw}
\|u(\cdot,t)-b\|_{L^\infty}+\|v(\cdot,t)-b\|_{L^\infty}+\|w(\cdot,t)\|_{W^{2,\infty}}\leq C_{14} e^{-\frac{\sigma}{2(n+1)} t}, \ \ \ t>0.
\end{equation}
\end{lemma}
\begin{proof}
By the estimate for $M_0$ in \eqref{b-v1} or \eqref{v-infty-bdd}, we first infer that
$$
M_0^2\leq O(1)\left(1+\left(\frac{1}{\mu}\right)^\frac{4+n}{\theta}\right),
$$
and then, we  find that our assumption on $\mu$ in \eqref{mu-large-lt-intro} along with \eqref{lambda-def}  implies
\begin{equation}\label{mu-large-lt}
\mu>\frac{\left[ \left(1+\frac{C^2_p\xi_2^2M_0^2}{d}\right)\frac{\chi^2}{d}+ C_p^2\xi_1^2\right]}{2b^{\theta-2}}\Longleftrightarrow\mu>\left(\frac{ d\chi^2 +d^2C^2_p\xi_1^2+ C^2_p\xi_2^2M_0^2}{2d^2a^\frac{\theta-2}{\theta}}\right)^\frac{\theta}{2}.
\end{equation}
Now, multiplying \eqref{l2v-dev-p} by  $\frac{b\chi^2}{2d}$ and adding the resulting inequality to \eqref{lg-u1}, we obtain
\begin{equation}\label{ulnu-v-con}
\begin{split}
&\frac{d}{dt}\left(\int_\Omega (u-b-b\ln \frac{u}{b})+\frac{b\chi^2}{2d}\int_\Omega(v-b)^2\right) +\frac{b\chi^2}{2d}\int_\Omega(v-b)^2\\
&\leq -\mu\int_\Omega (u-b)(u^\theta-b^\theta)+ \frac{b}{2}\left[ \left(1+\frac{C^2_p\xi_2^2M_0^2}{d}\right)\frac{\chi^2}{d}+ C^2_p\xi_1^2\right]\int_\Omega (u-b)^2.
\end{split}
\end{equation}
Notice, since $\theta\geq 1$, by examining monotonicity, one easily sees that
$$
\sup_{z\in(0, b)\cup (b,\infty)}\frac{(z-b)^2}{(z-b)(z^\theta-b^\theta)}=\frac{1}{b^{\theta-1}}\Longleftrightarrow -(z-b)(z^\theta-b^\theta)\leq -b^{\theta-1}(z-b)^2,  \  \ \forall z>0.
$$
Applying this inequality to \eqref{ulnu-v-con}, we immediately arrive at
\begin{equation}\label{ulnu-v-con2}
\begin{split}
&\frac{d}{dt}\left(\int_\Omega (u-b-b\ln \frac{u}{b})+ \frac{b\chi^2}{2d} \int_\Omega(v-b)^2\right)+ \frac{b\chi^2}{2d} \int_\Omega(v-b)^2\\
&\ \ +  b^{\theta-1}\left[\mu -\frac{\left[ \left(1+\frac{C_p^2\xi_2^2M_0^2}{d}\right)\frac{\chi^2}{d}+ C_p^2\xi_1^2\right]}{2b^{\theta-2}}\right] \int_\Omega (u-b)^2\leq 0.
\end{split}
\end{equation}
We first integrate \eqref{ulnu-v-con2}  from $1$ to $t$ and then we use our derived  condition  \eqref{mu-large-lt} and  the elementary algebraic  fact that $z-b-b\ln\frac{z}{b}\geq 0$ for $z>0$  to obtain that
\begin{equation}\label{u-b-con}
\int_1^t\int_\Omega (u-b)^2<c_{81}, \ \ \forall t\geq 1.
\end{equation}
Given \eqref{u-b-con}, using the same arguments as in Lemma \ref{K*} for $u$, one can  easily show that
\begin{equation}\label{cv-u8}
\|u(\cdot,t)-b\|_{L^\infty} \to 0\  \ \mathrm{as} \ t\to\infty.
 \end{equation}
Next, observing  the  following elementary fact that
$$
\lim_{z\rightarrow b}\frac{(z-b)^2}{z-b-b\ln \frac{z}{b}}=2b,
$$
we  infer from \eqref{cv-u8}  there exists $t_2>0$ such that
\begin{equation}\label{cv-v2}
b\left[u(\cdot, t)-b-b\ln \frac{u(\cdot, t)}{b}\right]\leq \left[u(\cdot, t)-b\right]^2\leq 3b\left[u(\cdot, t)-b-b\ln \frac{u(\cdot, t)}{b}\right], \  \forall t\geq t_2.
\end{equation}
Substituting \eqref{cv-v2} into \eqref{ulnu-v-con2}, we end up with an important ODI:
\begin{equation}\label{ulnu-v-con3}
\begin{split}
&\frac{d}{dt}\left(\int_\Omega (u-b-b\ln \frac{u}{b})+ \frac{b\chi^2}{2d}\int_\Omega(v-b)^2\right)\\
&\ \ +\sigma \left(\int_\Omega (u-b-b\ln \frac{u}{b})+ \frac{b\chi^2}{2d}\int_\Omega(v-b)^2\right)\leq 0, \ \ \ \forall t\geq t_2,
\end{split}
\end{equation}
where, due to \eqref{mu-large-lt},
$$
\sigma:=\min\left\{1, \ \ \   b^\theta \left[\mu -\frac{\left[ \left(1+\frac{C^2_p\xi_2^2M_0^2}{d}\right)\frac{\chi^2}{d}+ C^2_p\xi_1^2\right]}{2b^{\theta-2}}\right]\right\}>0.
$$
Solving the ODI \eqref{ulnu-v-con3}, we simply get the following exponential decay estimate:
\begin{equation*}
\int_\Omega (u-b-b\ln \frac{u}{b})+ \int_\Omega(v-b)^2\leq c_{82} e^{-\sigma (t-t_2)}, \ \ \ \forall t\geq t_2,
\end{equation*}
which combined with \eqref{cv-v2} gives
\begin{equation}\label{cv-v5*}
\|u(\cdot,t)-b\|_{L^2}+\|v(\cdot,t)-b\|_{L^2}\leq c_{83}e^{-\frac{\sigma}{2} (t-t_2)}, \ \ \ \forall t\geq t_2.
\end{equation}
Now, it is easy to improve this $L^2$-convergence to $L^\infty$-convergence as in  Lemma \ref{K*}. Indeed, by the  boundedness of $\|\nabla u\|_{L^{2n}}$ in \eqref{up-1} and the GN inequality,  \eqref{cv-v5*} entails
\begin{equation} \label{u-conv-fin}
\begin{split}
    \|u(\cdot,t)-b\|_{L^\infty}&\leq c_{84}\left(\|\nabla u(\cdot,t)\|^{\frac{n}{n+1}}_{L^{2n}}\|u(\cdot,t)-b\|^{\frac{1}{n+1}}_{L^2}
    +\|u(\cdot,t)-b\|_{L^2}\right)\\
    &\leq c_{85}e^{-\frac{\sigma}{2(n+1)}t} , \ \ \ \forall t\geq t_2.
    \end{split}
\end{equation}
Similarly, using the $W^{1,\infty}$-boundedness of $v$ guaranteed by Theorem \ref{GB}, we use the GN inequality to improve the $L^2$-convergence of $v$ in \eqref{cv-v5*} to its $L^\infty$-convergence as follows:
\begin{equation} \label{v-conv-fin}
\begin{split}
    \|v(\cdot,t)-b\|_{L^\infty}\leq c_{86}\| v(\cdot,t)\|_{W^{1,\infty}}^\frac{n}{n+2}\|v(\cdot,t)-b\|_{L^2}^\frac{2}{n+2}
    &\leq c_{87}e^{-\frac{\sigma}{n+2}(t-t_2)} , \ \ \ \forall t\geq t_2.
    \end{split}
\end{equation}
As for the $W^{2,\infty}$-convergence of $w$, applying  the  $W^{2,\infty}$-estimate  to \eqref{w-2p} and using \eqref{u-conv-fin}, we see that $w$ decays to zero exponentially in $W^{2,\infty}$-norm:
\begin{equation} \label{w-conv-fin}
\begin{split}
\|w(\cdot, t)\|_{W^{2,\infty}}&\leq c_{88}\|\Delta w\|_{L^\infty}\leq c_{88}\left(\|u(\cdot, t)-b\|_{L^\infty}+|b-\bar{u}|\right)\\
&\leq 2c_{88}\|u(\cdot, t)-b\|_{L^\infty}\leq c_{89}e^{-\frac{\sigma}{2(n+1)}t} , \ \ \ \forall t\geq t_2.
    \end{split}
\end{equation}
Our desired  exponential convergence \eqref{cv-uvw} follows from \eqref{u-conv-fin}, \eqref{v-conv-fin} and \eqref{w-conv-fin}.
\end{proof}

\begin{proof}[{\bf Proof of Theorem \ref{LT}}]
 The respective of convergence and exponential convergence in (C1) and (C2) of Theorem \ref{LT} have been fully contained in Lemmas \ref{K*}, \ref{vc}, and \ref{cv-u}.
\end{proof}

\bigbreak
\noindent \textbf{Acknowledgment.}
 H. Y. Jin was supported by the Guangdong Basic and Applied Basic Research Foundation (No.
2020A1515010140and 2020B1515310015),Guangzhou Science and Technology Pro-
gram (No. 202002030363), NSF of China (No. 11871226 and No.12026608), and the
Fundamental Research Funds for the Central Universities.   T. Xiang  was  funded by the NSF of China (No. 12071476  and 11871226) and  the Research Funds  of Renmin University of China (No. 2018030199).

\bigbreak
\bigbreak


\begin{thebibliography}{99}
\bibitem{ADN59} S. Agmon, A. Douglis and L. Nirenberg, Estimates near the boundary for solutions of elliptic partial differential equations satisfying general boundary conditions. I. {\it Commun. Pure Appl. Math.}, 12: 623--727, 1959.

\bibitem{ADN64} S. Agmon, A. Douglis and L. Nirenberg, Estimates near the boundary for solutions of elliptic partial differential equations satisfying general boundary conditions. II. {\it Commun. Pure Appl. Math.}, 17: 35--92, 1964.
\bibitem{Am-1990} H. Amann, Dynamic theory of quasilinear parabolic equations. II. Reaction-diffusion systems. {\it Differ. Integral Equ.,}
3(1): 13-75, 1990.

\bibitem{Am-1993}H. Amann, Nonhomogeneous linear and quasilinear elliptic and parabolic boundary value problems. {\it Function Spaces, Differential Oparators and Nonlinear Analysis, Teubner Texte zur Mathematik.,} 133:9--126, 1993.

\bibitem{AT21-AAM} G. Arumugam and J. Tyagi, Keller-Segel chemotaxis models: a review. {\it Acta. Appl. Math. ,} 171, Paper No. 6, 82 pp,  2021.

\bibitem{BBTW-2015} N. Bellomo, A. Bellouquid, Y.S. Tao and M. Winkler, Toward a mathematical theory of Keller-Segel models of pattern formation in biological tissues. {\it Math. Models Methods Appl. Sci.,} 25:1663--1763, 2015.

\bibitem{Cao15} X. Cao, Global bounded solutions of the higher-dimensional Keller-Segel system under smallness conditions in optimal spaces. {\it Discrete Contin. Dyn. Syst.,}  35: 1891--1904, 2015.



\bibitem{CS-IMA-1993} M.A.J. Chaplain and A.M. Stuart,  A model mechanism for the chemotactic response of endothelial cells to tumour  angiogenesis factor. {\it IMA J. Math. Appl. Med. Biol.,} 10:149--168, 1993.


\bibitem{Espejo2014}
E. Espejo and T. Suzuki, Global existence and blow-up for a system describing the aggregation of microglia. {\it Appl. Math. Lett.}, 35:29--34, 2014.


\bibitem{Horstmann-D} D. Horstmann, From 1970 until present: the Keller-Segel model in chemotaxis and its consequences. I. {\it{ Jahresber. Deutsch. Math.-Verien.,}} 105:103--165, 2003.
\bibitem{horstmann2001blow} D. Horstmann and G. Wang, Blow-up in a chemotaxis model without symmetry
assumptions. {\it{Eur. J. Appl. Math.,}} 12:159--177, 2001.
\bibitem{JL-TAMS-1992} W. J\"{a}ger and S. Luckhaus, On explosions of solutions to a system of partial differentail equations modelling chemotaxis. {\it  Trans. Amer. Math. Soc.,} 329:819-824, 1992.

\bibitem{J-JMAA-2015}H.Y. Jin, Boundedness of the attraction-repulsion Keller-Segel system. {\it J. Math. Anal. Appl.}, 422:1463--1478, 2015.
%\bibitem{JL-AML-2015} H.Y. Jin and Z. Liu, Large time behavior of the full attraction-repulsion Keller-Segel system in the whole space. {\it Appl. Math. Lett.}, 47:13-20, 2015.

\bibitem{JW-M2AS-2015}H.Y. Jin and Z.A. Wang, Asymptotic dynamics of the one-dimensional attraction-repulsion Keller-Segel model. {\it Math. Methods Appl. Sci.}, 38:444--457, 2015.


\bibitem{JW-JDE-2016} H.Y. Jin and Z.A. Wang, Boundedness, blowup and critical mass phenomenon in competing chemotaxis. {\it J. Diff. Eqns.}, 260:162--196, 2016.


\bibitem{JX18}H.Y. Jin and T. Xiang, Chemotaxis effect vs logistic  damping  on boundedness for the 2-D minimal Keller-Segel model. {\it  C. R. Acad. Sci. Paris, Ser. I},  356: 875--885, 2018.

\bibitem{JW-preprint}
H.Y. Jin and Z.A. Wang, Global stabilization of the full attraction-repulsion Keller-Segel system. {\it Discrete Contin. Dyn. Syst.}, 40:3509--3527, 2020.


\bibitem{JX-CMS-2021} H.Y. Jin and J. Xu,
Analysis of the role of convection in a system describing the tumor-induced angiogenesis. {\it Comm. Math. Sci.,}  19:1033--1049, 2021.

\bibitem{La} O. Ladyzhenskaya, V. Solonnikov and N. Uralceva, {\it Linear and Quasilinear Equations of Parabolic Type}. AMS, Providence, RI, 1968.

\bibitem{LL-NARWA-2016} Y. Li and Y.X. Li, Blow-up of nonradial solutions to attraction-repulsion chemotaxis system in two dimensions. {\it Nonlinear Anal. Real Word Appl.,} 30:170--183, 2016.

\bibitem{Tao-JMAA}G. Li and Y.S. Tao, Analysis of a chemotaxis-convection model of capillary-sprout growth during tumor angiogenesis. {\it  J. Math. Anal. Appl.,} 481,123474, 2020.

%\bibitem{LM-NAR-2016}K. Lin and C. Mu, Global existence and convergence to steady states for an attraction-repulsion chemotaxis system. {\it  Nonlinear Anal. Real Word Appl.,} 31:630--642, 2016.
%\bibitem{LMW-JMAA-2015} K. Lin, C. Mu and L. Wang, Large-time behavior of an attraction-repulsion chemotaxis system.   {\it J. Math. Anal. Appl.}, 426:105--124, 2015.
\bibitem {Lin-M3AS-2018}  K. Lin, C. Mu and D. Zhou, Stabilization in a higher-dimensional attraction-repulsion chemotaxis system if repulsion dominates over attraction. {\it Math. Models Methods Appl. Sci.,} 28:1105--1134, 2018.



\bibitem{LX-20-CVPDE} K. Lin and T.  Xiang,  On boundedness, blow-up and convergence in a two-species and two-stimuli chemotaxis system with/without loop. {\it  Calc. Var. Partial Differential Equations}, 59   Paper No. 108, 35 pp., 2020.


\bibitem{LSW-DCDSB-2013}
    P. Liu, J.P. Shi and Z.A. Wang, Pattern formation of the attraction-repulsion Keller-Segel system. {\it Discrete Contin. Dyn. Syst. Ser. B.}, 18:2597--2625, 2013.

\bibitem{LT-M2AS-2015} D. Liu and Y.S. Tao, Global boundedness in a fully parabolic attraction-repulsion chemotaxis model.  {\it Math. Methods Appl. Sci.}, 38:2537--2546, 2015.


\bibitem{LW-JBD-2012} J. Liu and Z.A. Wang,  Classical solutions and steady states of an attraction-repulsion chemotaxis model in one dimension. {\it{J. Biol. Dyn.,}} 6:31--41, 2012.

\bibitem{M-A-L-A} M. Luca, A. Chavez-Ross, L. Edelstein-Keshet and A. Mogilner, Chemotactic signalling, Microglia, and Alzheimer's disease senile plagues: Is there a connection?  {\it{Bull. Math. Biol.,}} 65:693--730, 2003.

\bibitem{Ma-79} H. Matano, Asymptotic behavior and stability of solutions of semilinear diffusion equations. {\it Publ. Res. Inst. Math. Sci.}, 15:401--454, 1979.

\bibitem{MS-Poincare-2014}%MR3249815
N. Mizoguchi and P. Souplet, Nondegeneracy of blow-up points for the parabolic Keller-Segel system. {\it Ann. Inst. H. Poincar\'{e} Anal. Non Lin\'{e}aire},
31(4):851--875, 2014.
%\bibitem{Murray} J.D. Murray, Mathematical Biology Springer Verlag Berlin. {\it Heidelberg, New York}, 1989.
 \bibitem{Nagai-Funk}T. Nagai, T. Senba and K. Yoshida, Application of the Trudinger-Moser inequality to a parabolic system of chemotaxis. {\it{Funkcial. Ekvac.,}} 40:411--433, 1997.
\bibitem{OC-IMA-1996} M.E. Orme and M.A.J. Chaplain, A mathematical model for the first steps of tumour-related angiogenesis: Capillary sprout formation and secondary branching. {\it IMA J. Math. Appl. Med. Biol.,} 13:73--98, 1996.


\bibitem{P-H-2002} K.J. Painter and T. Hillen, Volume-filling quorum-sensing in models for chemosensitive movement. {\it{Can. Appl. Math. Q.,}}  10:501--543, 2002.
%\bibitem{RT} M.M. Porzio and  V.Vespri. H\"{o}lder estimates for local solutions of some doubly nonlinear degenerate parabolic equations. {\it {J. Differential Equations,}} 103(1):146-178, 1993.
\bibitem{PK-CROH-1989} N. Paweletz and M. Knierim, Tumor-related angiogenesis. {\it Crit. Rev. Oncol. Hematol.,} 9:197--242, 1989.

\bibitem{Po93} M. M. Porzio, H\"older estimates for local solutions of some doubly nonlinear degenerate parabolic equations. {\it J. Diff. Eqns.}, 103:146-178, 1993.


\bibitem{SL-JTB-1991} C.L. Stokes and D.A. Lauffenburger, Analysis of the roles of microvessel endothelial cell random motility and chemotaxis in angiogenesis. {\it J. Theor. Biol.,} 152:377--403, 1991.

\bibitem{SQ07} P. Souplet and P. Quittner, {\it Superlinear parabolic problems: blow-up, global existence and steady states}. Birkh\"auser Advanced Texts, Basel/Boston/Berlin, 2007.




 \bibitem{Tao-Winkler-SIMA-2015} Y.S. Tao and M. Winkler, Large time behavior in a multidimensional chemotaxis-haptotaxis model with slow signal diffusion. {\it SIAM J. Math. Anal.}, 47:4229-4250, 2015.


\bibitem{TW-NA-2021} Y.S. Tao and M. Winkler, The dampening role of large repulsive convection in a chemotaxis system modeling tumor angiogenesis. {\it  Nonlinear Anal.},   112324,  2021.


 \bibitem{TW-M3AS-2013} Y.S. Tao and Z.A. Wang, Competing effects of attraction vs. repulsion in chemotaxis. {\it{Math. Models Methods Appl. Sci.,}} 23:1--36, 2013.



\bibitem{W-JDE-2010} M. Winkler, Aggregation vs. global diffusive behavior in the higher-dimensional Keller-Segel model. {\it J. Diff. Eqns.}, 248:2889--2905, 2010.


\bibitem{W-JMAA-2011}M. Winkler, Blow-up in a higher-dimensional chemotaxis system despite logistic growth restriction. {\it J. Math. Anal. Appl.,} 384:261-272, 2011.


\bibitem{W-JMPA-2013} M. Winkler, Finite-time blow-up in the higher-dimensional parabolic-parabolic Keller-Segel system. {\it J. Math. Pures Appl.},  100:748--767, 2013.


\bibitem{Xiangjde} T. Xiang, Boundedness and global existence in the higher-dimensional parabolic-parabolic chemotaxis system with/without growth source. {\it J. Diff. Eqns.},  258:  4275--4323,  2015.


\bibitem{Xiang18} T. Xiang,  Chemotactic aggregation versus logistic damping on boundedness in the 3D minimal Keller-Segel model.  {\it SIAM J. Appl. Math.}, 78: 2420--2438, 2018.


\end{thebibliography}
\end{document}